%% file: death-birth.tex
\documentclass[letterpaper,11pt,twoside,keywordsasfootnote,addressatend,noinfoline]{article}
\usepackage{fullpage}
\usepackage[english]{babel}
\usepackage{amssymb}
\usepackage{amsmath}
\usepackage{theorem}
\usepackage{epsfig}
\usepackage{subfigure}
\usepackage{mathrsfs}
\usepackage{imsart}
\usepackage{color}

\theorembodyfont{\normalfont}
\newtheorem{theorem}{Theorem}

\newtheorem{lemma}[theorem]{Lemma}

\newenvironment{proof}{\noindent{\scshape Proof.}}{\hspace*{2mm}~$\square$}

\newcommand{\N}{\mathbb{N}}
\newcommand{\Z}{\mathbb{Z}}
\newcommand{\R}{\mathbb{R}}
\renewcommand{\H}{\mathscr{H}}
\newcommand{\ind}{\mathbf{1}}
\newcommand{\ep}{\epsilon}
\newcommand{\floor}[1]{\lfloor{#1}\rfloor}
\newcommand{\n}{\hspace*{-5pt}}

\DeclareMathOperator{\card}{card}
\DeclareMathOperator{\binomial}{Binomial \,}

\begin{document}

\begin{frontmatter}
\title     {Evolutionary games on the lattice: \\ death-birth updating process}
\runtitle  {Death-birth updating process}
\author    {Stephen Evilsizor and Nicolas Lanchier\thanks{Both authors were supported in part by NSA Grant MPS-14-040958.}}
\runauthor {S. Evilsizor and N. Lanchier}
\address   {School of Mathematical and Statistical Sciences \\ Arizona State University \\ Tempe, AZ 85287, USA.}

\maketitle

\begin{abstract} \ \
 This paper is concerned with the death-birth updating process.
 This model is an example of a spatial game in which players located on the~$d$-dimensional integer lattice are characterized by one of two possible strategies and update their strategy at rate one by mimicking one of their
 neighbors chosen at random with a probability proportional to the neighbor's payoff.
 To understand the role of space in the form of local interactions, the process is compared with its nonspatial deterministic counterpart for well-mixing populations, which is described by the replicator equation.
 To begin with, we prove that, provided the range of the interactions is sufficiently large, both strategies coexist on the lattice for a parameter region where the replicator equation also exhibits coexistence.
 Then, we identify parameter regions in which there is a dominant strategy that always wins on the lattice whereas the replicator equation displays either coexistence or bistability.
 Finally, we show that, for the one-dimensional nearest neighbor system and in the parameter region corresponding to the prisoner's dilemma game, cooperators can win on the lattice whereas defectors always win in well-mixing
 populations, thus showing that space favors cooperation.
 In particular, several parameter regions where the spatial and nonspatial models disagree are identified.
\end{abstract}



\end{frontmatter}


\section{Introduction}
\label{sec:intro}

\indent This paper is concerned with a closely related version of the death-birth updating process in evolutionary game theory introduced in~\cite{ohtsuki_al_2006}.
 This model is an example of a spatial evolutionary game designed based on the framework of interacting particle systems.
 The concept of evolutionary game theory is an extension of traditional game theory that has been proposed by~\cite{maynardsmith_price_1973} to describe the dynamics of populations in which fitness is frequency dependent:
 individuals are viewed as players who are characterized by their strategy and receive a certain payoff through their interactions with other individuals.
 The payoff is then interpreted as fitness so that players with a larger payoff have a higher reproductive success.
 One of the most popular models of evolutionary game is the replicator equation~\cite{hofbauer_sigmund_1998} which assumes that the population is well-mixing.
 In contrast, the death-birth updating process includes a spatial structure in the form of local interactions:
 players are located on the set of vertices of a connected graph and interact with a finite set of neighbors, meaning that they update their strategy based on these neighbors.
 See~\cite[chapter~9]{nowak_2006} for a general definition and a brief review of such models.
 The model in~\cite{ohtsuki_al_2006} assumes that the updates are neutral with high probability and based on the payoff of the neighbors with small probability, which we refer respectively as voter and game steps.
 In contrast, the model considered in this paper only accounts for game steps, so the duality techniques~\cite{cox_durrett_perkins_2012} developed for voter model perturbations are no longer available tools to study the process.
 The main objective is to study the limiting behavior of the spatial stochastic process and confront our results with the limiting behavior of the replicator equation in order to understand the effects of the inclusion of space. \vspace*{5pt}

\pagebreak


\noindent{\bf Model description.}
 The process studied in this paper, which we again refer to as the death-birth updating process following the terminology in~\cite{ohtsuki_al_2006}, is a spin system on the~$d$-dimensional integer lattice where each vertex is occupied
 by a player characterized by one of two possible strategies, say strategy~1 and strategy~2.
 The state at time~$t$ is a function 
 $$ \xi_t : \Z^d \to \{1, 2 \} \quad \hbox{where} \quad \xi_t (x) = \hbox{strategy at vertex~$x$ at time~$t$}. $$
 The dynamics of this process or any other spatial game is defined in a couple of steps:
 we first fix a payoff matrix, which allows us to turn every spatial configuration of strategies into a so-called payoff landscape, which can then be used to define the transition rates at each vertex.
 Since we focus on games with two strategies, the payoff matrix is a~$2 \times 2$ matrix~$A = (a_{ij})$ whose coefficients are positive real numbers interpreted as
 $$ a_{ij} = \hbox{payoff of a type~$i$ player interacting with a type~$j$ player}. $$
 In nonspatial evolutionary games, players interact equally with any other player in the population, making their payoff a function of the global frequency of representatives of each strategy.
 In contrast, spatial games assume that the payoff of a player depends exclusively on the strategy of a finite set of neighbors, which is the key to designing more realistic models with local interactions.
 Throughout this paper, the interaction neighborhood of vertex~$x$ is the set
 $$ \begin{array}{l} N_x := \{y \in \Z^d : y \neq x \ \hbox{and} \ \max_{j = 1, 2, \ldots, d} |x_j - y_j| \leq M \} \end{array} $$
 where the constant~$M$ is called the range of the interactions.
 Letting~$N_j (x, \xi)$ be the number of neighbors of the player at vertex~$x$ following strategy~$j$, every spatial configuration~$\xi$ is then turned into a payoff landscape by attributing the payoff
\begin{equation}
\label{eq:landscape}
  \begin{array}{l} \phi (x, \xi) := \sum_j \,a_{ij} \,N_j (x, \xi) \,\ind \{\xi (x) = i \} \quad \hbox{for} \quad i = 1, 2, \end{array}
\end{equation}
 to the player at vertex~$x$.
 In words, each type~$i$ player receives~$a_{ij}$ from each of her neighbors following strategy~$j$.
 The last step to define the dynamics of the process is to follow~\cite{maynardsmith_price_1973} and interpret the payoff as fitness.
 The basic idea here is to write the rate at which a player changes her strategy as a function of her payoff and the payoff of her neighbors in such a way that players with a larger payoff are more likely to spread their strategy.
 There are multiple options.
 For instance, the updating rules considered in~\cite{evilsizor_lanchier_2014, lanchier_2015} are as follows.
\begin{itemize}
 \item Best-response dynamics~\cite{evilsizor_lanchier_2014}.
       Players update their strategy at rate one in order to maximize their payoff, which depends on the strategy of their neighbors. \vspace*{4pt}
 \item Payoff affecting birth and death rates~\cite{lanchier_2015}.
       In this process, when a player has a positive payoff, at rate this payoff, one of her neighbors chosen at random adopts her strategy, whereas when her payoff is negative, at rate minus this payoff,
       she adopts the strategy of one of her neighbors chosen at random.
       This updating rule is inspired from~\cite{brown_hansell_1987}.
\end{itemize}
 The dynamics of the death-birth updating process is built using a similar approach:
 we assume that players update their strategy at rate one by mimicking a random neighbor, with each neighbor being chosen with a probability proportional to her payoff.
 More precisely, letting~$\xi$ be the configuration of the system, the player at~$x$ switches her strategy~$i \to j$ at rate
\begin{equation}
\label{eq:transition}
  p_{i \to j} (x, \xi) := \frac{\sum_{y \in N_x} \phi (y, \xi) \ \ind \{\xi (y) = j \}}{\sum_{y \in N_x} \phi (y, \xi)} \quad \hbox{for} \quad \{i, j \} = \{1, 2 \}.
\end{equation}
 This model, or to be more specific, a closely related version of this model, has been introduced and studied heuristically in~\cite{ohtsuki_al_2006} while~\cite{chen_2013, cox_durrett_perkins_2012} give rigorous results.
 The process considered in these works can be seen as the weak selection approximation of the model described by~\eqref{eq:transition}.
 Players again update their strategy at rate one but, at the time of the update,
\begin{itemize}
 \item with probability~$1 - \ep$, the player mimics the strategy of a neighbor chosen uniformly at random, just like in the voter model~\cite{clifford_sudbury_1973, holley_liggett_1975}, while \vspace*{4pt}
 \item with probability~$\ep$, the player mimics a neighbor chosen at random according to probabilities that are proportional to the neighbors' payoff, as described by~\eqref{eq:transition}.
\end{itemize}
 This model is studied in~\cite{chen_2013, cox_durrett_perkins_2012, ohtsuki_al_2006} when~$\ep$ is small, in which case duality techniques for voter model perturbations are available.
 In contrast, we study the process when~$\ep = 1$, in which case duality cannot be used, which leads to more qualitative and less quantitative results. \vspace*{5pt}


\noindent{\bf The replicator equation.}
 Before studying the spatial game, it is worth taking a quick look at its nonspatial deterministic analog to later identify disagreements between both models and thus understand the effect of the inclusion of space in the form
 of local interactions.
 The nonspatial model is obtained by assuming that the population of players is well-mixing, which results in a system of ordinary differential equations for the frequency of each strategy.
 In the case of the death-birth process, this is a time-change of the replicator equation:
\begin{equation}
\label{eq:replicator}
  u_1' = u_1 \,u_2 \,(\phi_1 (u_1, u_2) - \phi_2 (u_1, u_2))
\end{equation}
 where~$u_j$ is the frequency of players following strategy~$j$ and
 $$ \phi_1 (u_1, u_2) = a_{11} u_1 + a_{12} u_2 \quad \hbox{and} \quad \phi_2 (u_1, u_2) = a_{21} u_1 + a_{22} u_2 $$
 are the common payoffs of all type~1 and all type~2 players, respectively.
 This can be viewed as the nonspatial analog of the payoff landscape~\eqref{eq:landscape}.
 As pointed out in~\cite{lanchier_2013}, the limiting behavior can be conveniently described by introducing the parameters
 $$ a_1 := a_{11} - a_{21} \quad \hbox{and} \quad a_2 := a_{22} - a_{12} $$
 and calling strategy~$i$ selfish whenever~$a_i > 0$ and altruistic whenever~$a_i < 0$.
 Then, following the usual terminology by calling a strategy an evolutionary stable strategy if it cannot be invaded by any alternative strategy starting at an infinitesimally small frequency,
 some basic algebra shows that the behavior of the replicator equation~\eqref{eq:replicator} is as follows:
\begin{itemize}
 \item when~$a_1 \,a_2 < 0$, the selfish strategy always outcompetes the altruistic strategy, showing that the selfish strategy is the only evolutionary stable strategy, \vspace*{4pt}
 \item when~$a_1, a_2 > 0$, there is an unstable interior fixed point so the system is bistable, showing that the two (selfish) strategies are evolutionary stable, \vspace*{4pt}
 \item when~$a_1, a_2 < 0$, there is a globally stable interior fixed point so both strategies coexist and none of the two (altruistic) strategies is evolutionary stable.
\end{itemize}
 In summary, the analysis of the replicator equation shows that, when the population is well-mixing, a strategy is an evolutionary stable strategy if it is selfish but not if it is altruistic. \vspace*{5pt}


\noindent{\bf Main results for the spatial game.}
 In order to compare the spatial game with its nonspatial analog, we assume that the process starts from a spatially homogeneous distribution, i.e., a product measure in which the density of each of the two strategies
 is constant across the lattice.
 Since the two configurations in which all players follow the same strategy are absorbing states, we also assume, to avoid trivialities, that the density of each strategy is positive.
 For the spatial game,
\begin{itemize}
 \item strategy~$j$ wins when~$\lim_{t \to \infty} \,P \,(\xi_t (x) = j) = 1$, \vspace*{4pt}
 \item strategy~$j$ survives when~$\liminf_{t \to \infty} \,P \,(\xi_t (x) = j) > 0$,
\end{itemize}
 a strategy is said to go extinct when it does not survive and both strategies are said to coexist when they both survive.
 Note that the two probabilities above do not depend on the choice of vertex~$x$ because both the initial distribution and the evolution rules are translation invariant.
 From now on, we assume without loss of generality that~$a_{21} > a_{12} > 0$ and study the limiting behavior of the process as the other two payoffs vary.

\indent To begin with, we look at the parameter region where both strategies are altruistic, in which case coexistence occurs when the population is well-mixing.
 Numerical simulations suggest that, except in the one-dimensional nearest neighbor case, coexistence is again possible for the spatial game though the coexistence region is reduced.
 The smaller the spatial dimension and/or the range of the interactions, the smaller the coexistence region.
 Our first two theorems show that coexistence is indeed possible and that the coexistence region for the spatial game is indeed reduced.
 More precisely, Theorem~\ref{th:coex} shows that, regardless of the spatial dimension, both strategies coexist when they are sufficiently altruistic and the range of the interactions is sufficiently large.
\begin{theorem} --
\label{th:coex}
 There is~$a > 0$ such that coexistence occurs when
 $$ \max \,(a_{11}, a_{22}) \leq a \quad \hbox{and} \quad M \ \hbox{is sufficiently large}. $$
\end{theorem}
 To prove that the coexistence region is reduced, and more generally identify parameter regions in which one strategy wins, we first observe that, when~$a_{11} = a_{12}$ and~$a_{22} = a_{21}$, the process is significantly
 simplified because the payoff of the players only depends on their strategy but not on the strategy of their neighbors.
 Indeed, in this case~\eqref{eq:landscape} reduces to
 $$ \begin{array}{l}
    \phi (x, \xi) = \sum_j a_{ij} \,N_j (x, \xi) \,\ind \{\xi (x) = i \} = a_{ii} \,((2M + 1)^d - 1) \,\ind \{\xi (x) = i \} \end{array} $$
 for~$i = 1, 2$, therefore the transition rates become
 $$ p_{i \to j} (x, \xi) = \frac{\sum_{y \in N_x} \phi (y, \xi) \ \ind \{\xi (y) = j \}}{\sum_{y \in N_x} \phi (y, \xi)} = \frac{a_{jj} \,N_j (x, \xi)}{a_{ii} \,N_i (x, \xi) + a_{jj} \,N_j (x, \xi)} $$
 for~$\{i, j \} = \{1, 2 \}$.
 It follows that, under our general assumption~$a_{12} < a_{21}$, the set of type~2 players dominates stochastically a certain biased voter model~\cite{bramson_griffeath_1980, bramson_griffeath_1981}, thus showing that,
 in this very special case, strategy~2 wins.
 Elaborating on this idea but using coupling arguments to compare the death-birth updating process with spin systems which are more complicated than the biased voter model, we can prove much more, as shown in the next theorem.
\begin{theorem} --
\label{th:coupling}
 Recall that~$a_{21} > a_{12}$ and let~$N := \card N_x$. Then,
\begin{itemize}
 \item[(a)] strategy~1 wins when~$\min \,(a_{12} - a_{22}, a_{11} - a_{21}) > (N - 1)(a_{21} - a_{12})$, \vspace*{4pt}
 \item[(b)] strategy~2 wins when~$(M, d) \neq (1, 1)$ and
  $$ (N^2 - N - 1) \,\max \,(a_{11} - a_{21}, a_{12} - a_{22}, a_{11} - a_{22}) < a_{21} - a_{12}. $$
\end{itemize}
\end{theorem}
 Figure~\ref{fig:diagram-2D} shows the parameter regions in both parts of the theorem.
 Note that the parameter region in the first part of the theorem is nonempty if and only if
 $$ a_{12} - (N - 1)(a_{21} - a_{12}) > 0 \quad \hbox{if and only if} \quad a_{12} > (1 - 1/N) \,a_{21}. $$
 The figure shows this region for~$N = 2$, i.e., in the presence of one-dimensional nearest neighbor interactions.
 In contrast, the parameter region in the second part of the theorem is always nonempty and, more interestingly, it always overlaps the region where both strategies are altruistic as well as the region where both strategies are selfish.
 This shows that the inclusion of a spatial structure in the form of local interactions indeed reduces the coexistence region, as mentioned above.
 This also shows that, in a subset of the parameter region where the replicator equation is bistable, there is instead a strong type for the spatial game that wins even starting at low density.
 The next theorem goes a little bit further in this direction by showing that, no matter how selfish a strategy is, the other strategy always wins if it is selfish enough.
\begin{theorem} --
\label{th:growth}
 For all~$a > 0$ there is~$A < \infty$ such that strategy~1 wins when
 $$ \max \,(a_{21}, a_{22}) \leq a \quad \hbox{and} \quad a_{11} > A. $$
\end{theorem}
 Note that, contrary to Theorem~\ref{th:coupling}, this theorem does  not assume that~$a_{12} < a_{21}$, therefore it also extends the parameter region where strategy~2 wins found before.

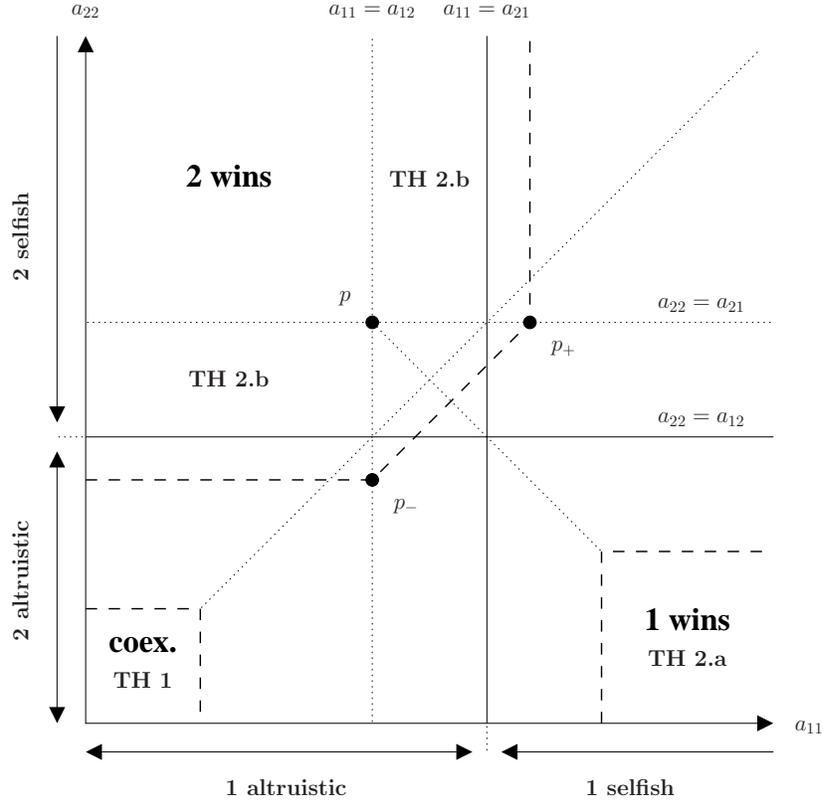
\begin{figure}[t]
 \centering
 \scalebox{0.60}{\input{diagram-2D.pstex_t}}
 \caption{\upshape Phase diagram of the spatial game along with a summary of the theorems in the $a_{11} - a_{22}$ plane.
  In the picture, the points~$p_-$ and~$p_+$ are the two points introduced in the proof of Lemma~\ref{lem:region}.}
\label{fig:diagram-2D}
\end{figure}
\begin{figure}[t]
 \centering
 \scalebox{0.60}{\input{diagram-1D.pstex_t}}
 \caption{\upshape Phase diagram of the one-dimensional nearest neighbor spatial game when~$a_{21} / a_{12} = 2$.
  The curves are obtained from the conditions in Theorem~\ref{th:dilemma} and their symmetric expressions.
  The triangle in solid lines represents the parameter region corresponding to the prisoner's dilemma game.}
\label{fig:diagram-1D}
\end{figure}
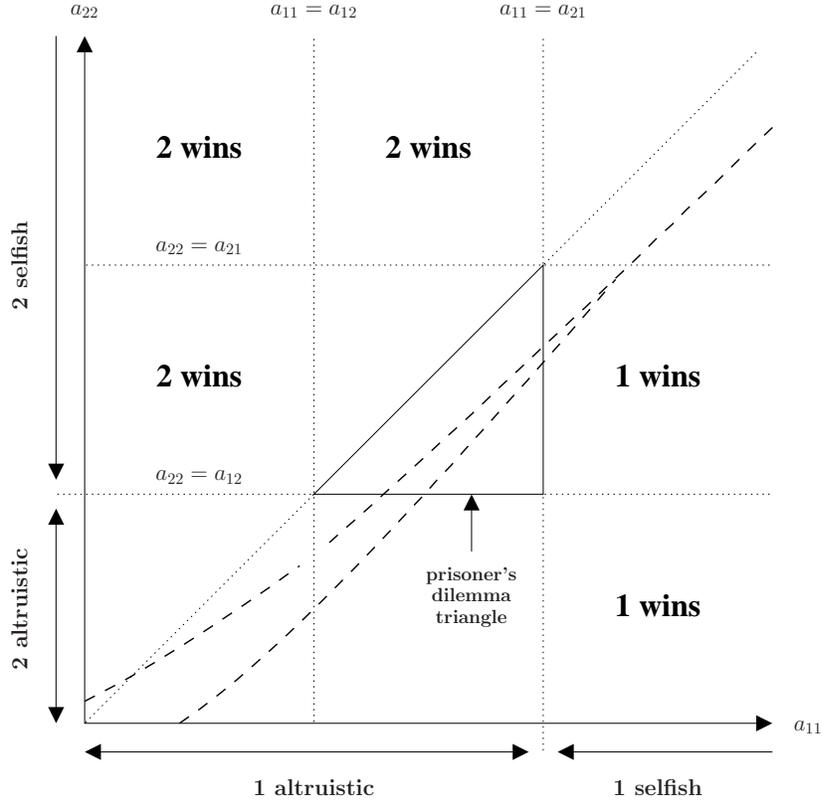

\indent The results collected so far indicate interesting discrepancies between the death-birth updating process and the replicator equation, showing the importance of local interactions.
 The most interesting aspect suggested by spatial simulations is the existence of a subset of the parameter region corresponding to the prisoner's dilemma game
 in which cooperators win on the lattice whereas they always lose when the population is well-mixing.
 The prisoner's dilemma game is characterized by the following ordering and terminology of the four payoffs:
 $$ \begin{array}{l} a_{12} = \hbox{sucker's payoff} < a_{22} = \hbox{punishment} \vspace*{2pt} \\ \hspace*{80pt} < a_{11} = \hbox{reward} < a_{21} = \hbox{temptation}. \end{array} $$
 Figure~\ref{fig:diagram-1D} shows the corresponding triangular region in solid lines.
 Players with strategy~1 are called cooperators while players with strategy~2 are called defectors.
 Because the reward is not as good as the temptation, and the punishment is not as bad as the sucker's payoff, cooperators are altruistic and defectors selfish, therefore defectors indeed win when the population is well-mixing.
 In contrast, the heuristic arguments in~\cite{ohtsuki_al_2006} suggest that there is a subset of the prisoner's dilemma triangle in which cooperators are favored over defectors on regular graphs.
 This has been proved in~\cite{chen_2013} for finite, connected, simple graphs, and in~\cite{cox_durrett_perkins_2012} for integer lattices with~$d > 2$.
 Their results, however, hold in the weak selection case but not for the process~\eqref{eq:transition}.
 We now study the interactions among cooperators and defectors in one dimension, the main difficulty being the lack of attractiveness of the process.
 To state our last result, we introduce the following quantities that will be interpreted later as drift of a certain interface:
 $$ \begin{array}{rcl}
      D_3 & \n := \n & \displaystyle \frac{a_{11} + a_{12}}{a_{11} + a_{12} + a_{21} + a_{22}} - \frac{a_{21} + a_{22}}{2 a_{11} + a_{21} + a_{22}} \vspace*{10pt} \\
      D_4 & \n := \n & \displaystyle \frac{a_{11} + a_{12}}{a_{11} + a_{12} + 2 a_{22}} - \frac{a_{21} + a_{22}}{2 a_{11} + a_{21} + a_{22}}. \end{array} $$
 Then, we have the following theorem.
\begin{theorem} --
\label{th:dilemma}
 Assume~$M = d = 1$.
 Then, strategy~1 wins when
 $$ (a_{22} < a_{21} \ \hbox{and} \ D_3 + D_4 > 0) \quad \hbox{or} \quad (a_{22} > a_{21} \ \hbox{and} \ D_4 > 0). $$
\end{theorem}
 Note that the parameter region given by the theorem overlaps but is not restricted to the prisoner's dilemma triangle.
 To see that the theorem implies the existence of a subset of the triangle in which cooperators win, observe that, when~$a_{11} = a_{21} > a_{22} = a_{12}$,
 $$ D_3 + D_4 = \frac{1}{2} + \frac{a_{12} + a_{21}}{3 a_{12} + a_{21}} - 2 \times \frac{a_{12} + a_{21}}{a_{12} + 3 a_{21}} > \frac{1}{2} + \frac{1}{2} - 2 \times \frac{1}{2} = 0. $$
 In particular, the first parameter region given by the theorem in which cooperators win indeed overlap the prisoner's dilemma triangle.
 The theorem also implies that strategy~2 wins in the parameter regions obtained by exchanging the role of the two strategies.
 Noticing in addition the symmetry in the expression of~$D_4$, we deduce from the second part of the theorem that
 $$ \begin{array}{ccccc}
    \hbox{strategy~1 wins when} & a_{22} > a_{21} & \hbox{and} & D_4 > 0 \vspace*{2pt} \\
    \hbox{strategy~2 wins when} & a_{11} > a_{12} & \hbox{and} & D_4 < 0 \end{array} $$
 showing that, when~$\min (a_{11}, a_{22}) > \max (a_{12}, a_{21})$, the condition is sharp.
 Figure~\ref{fig:diagram-1D} gives a picture of the curves derived from the theorem when~$a_{21} / a_{12} = 2$.


\section{Coexistence of altruistic strategies}
\label{sec:coexistence}

\indent This section is devoted to the proof of our coexistence result Theorem~\ref{th:coex}.
 For simplicity, we focus on the two-dimensional case but our approach easily extends to any spatial dimensions.
 Specifically, we will prove that both strategies coexist when~$M$ is large and
\begin{equation}
\label{eq:condition-coex}
  \max \,(a_{11}, a_{22}) \leq 5^{-2} \,2^{-21} \,(c_-)^5 \,\min \,(a_{12}, a_{21}) = 2^{-14} \,(c_+)^{-1} \,\min \,(a_{12}, a_{21})
\end{equation}
 where the two key constants~$c_-$ and~$c_+$ are defined as
\begin{equation}
\label{eq:constant-coex}
  c_- := 2^{-17} \,\min \,(a_{12} / a_{21}, a_{21} / a_{12}) \quad \hbox{and} \quad c_+ := 5^2 \,2^7 (c_-)^{-5}.
\end{equation}
 Let~$s := \ln (2)$ and~$K_r := [- rM, rM)^2$ for all~$r > 0$, and fix
 $$ A, B \subset K_{1/2} \quad \hbox{with} \quad \card (A) = \card (B) = 2^{-2} \,M^2. $$
 The proofs of Lemmas~\ref{lem:move}--\ref{lem:create} below hold for such general sets though they will be applied ultimately to more specific space-time boxes.
 One key to the proof is to observe that
 $$ \begin{array}{l} \max_{j = 1, 2} |x_j - y_j| \leq M \quad \hbox{for all} \quad (x, y) \in A \times B. \end{array} $$
 For all~$D \subset \Z^2$ finite and~$i = 1, 2$, we let
 $$ \zeta_t^i (D) := \card \,\{x \in D : \xi_t (x) = i \} $$
 denote the number of type~$i$ players in the set~$D$ at time~$t$. \\


\noindent{\bf Keeping the players in a box} --
 To begin with, we prove in the next lemma that, the number of players of either type in a given spatial region does not decrease too fast.
 The idea is to simply find a bound for the number of updates using standard large deviation estimates for the binomial random variable.
 This lemma will be used repeatedly later.
\begin{lemma} --
\label{lem:keep}
 For all~$D \subset \Z^2$ finite, $n \in \N$ and~$i = 1, 2$,
 $$ P \,(\zeta_t^i (D) \leq 2^{- (n + 1)} \,K \ \hbox{for some} \ t \in (0, ns) \ | \ \zeta_0^i (D) \geq K) \leq \exp (- 2^{- (n + 3)} \,K). $$
\end{lemma}
\begin{proof}
 To begin with, we let
 $$ u_i (D) := \card \,\{x \in D : \xi_0 (x) = i \ \hbox{and} \ T_1 (x) < ns \} $$
 denote the total number of players in the set~$D$ who are initially of type~$i$ and update their strategy at least once by time~$ns$.
 Recalling that the random variables~$T_1 (x)$ are independent and exponentially distributed with rate one, our choice of~$s$ implies that
 $$ u_i (D) = \binomial (\zeta_0^i (D), 1 - e^{-ns}) = \binomial (\zeta_0^i (D), 1 - 2^{-n}). $$
 In particular, using the large deviation estimate
\begin{equation}
\label{eq:keep-1}
  P \,(\binomial (K, p) \leq K \,(p - z)) \leq \exp (- Kz^2 / 2p) \quad \hbox{for all} \quad z \in (0, p)
\end{equation}
 with~$p = 2^{-n}$ and~$z = 2^{- (n + 1)}$, we get
 $$ \begin{array}{l}
      P \,(\zeta_t^i (D) \leq 2^{- (n + 1)} \,K \ \hbox{for some time} \ t \in (0, ns) \ | \ \zeta_0^i (D) \geq K) \vspace*{4pt} \\ \hspace*{25pt} \leq \
      P \,(u_i (D) \geq (1 - 2^{- (n + 1)}) \,K \ | \ \zeta_0^i (D) = K) \vspace*{4pt} \\ \hspace*{25pt} = \
      P \,(\binomial (K, 1 - 2^{-n}) \geq (1 - 2^{- (n + 1)}) \,K) \vspace*{4pt} \\ \hspace*{25pt} = \
      P \,(\binomial (K, 2^{-n}) \leq 2^{- (n + 1)} \,K) \leq \exp (- 2^{- (n + 3)} \,K). \end{array} $$
 This completes the proof of the lemma.
\end{proof} \\


\noindent{\bf Moving the players around} --
 We now prove that if the region~$A$ has a large number of type~1 players then, regardless of the configuration around this region, we can ``move''  a positive fraction of these players to the nearby region~$B$ in~$s$ units of time.
 The constant~$c_-$ defined in~\eqref{eq:constant-coex} will play the role of the fraction of players we can move.
\begin{lemma} --
\label{lem:move}
 Assume~\eqref{eq:condition-coex} and let~$c_-$ as in~\eqref{eq:constant-coex} and~$a > 0$. Then,
 $$ P \,(\zeta_s^1 (B) \leq c_- (aM) \ | \ \zeta_0^1 (A) \geq aM) \leq \exp (- (aM)^{1/2}) \quad \hbox{for all~$M$ large}. $$
\end{lemma}
\begin{proof}
 Applying Lemma~\ref{lem:keep} with~$n = 1$ and~$K = aM$, we get
 $$ \begin{array}{l}
      P \,(\zeta_s^1 (B) \leq c_- (aM) \ | \ \zeta_0^1 (A) \geq aM \ \hbox{and} \ \zeta_0^1 (B) \geq aM) \vspace*{4pt} \\ \hspace*{25pt} \leq \
      P \,(\zeta_t^1 (B) \leq 2^{-2} (aM) \ \hbox{for some} \ t \in (0, s) \ | \ \zeta_0^1 (A) \geq aM \ \hbox{and} \ \zeta_0^1 (B) \geq aM) \vspace*{4pt} \\ \hspace*{25pt} \leq \
        \exp (- 2^{-4} \,aM) \leq \exp (- (aM)^{1/2}) \end{array} $$
 for all~$M$ large.
 It remains to prove that
\begin{equation}
\label{eq:move-1}
  P \,(\zeta_s^1 (B) \leq c_- (aM) \ | \ \zeta_0^1 (B) \leq aM \leq \zeta_0^1 (A)) \leq \exp (- (aM)^{1/2})
\end{equation}
 for all~$M$ large.
 To lighten the notation, we let
 $$ P^* (\mathbf E) := P \,(\mathbf E \ | \ \zeta_0^1 (B) \leq aM \leq \zeta_0^1 (A)) \quad \hbox{for any event} \ \mathbf E $$
 and introduce the two events
\begin{equation}
\label{eq:move-2}
  \begin{array}{rcl}
   \mathbf A & \n := \n & \{\zeta_t^1 (A) \leq 2^{-2} \,aM  \ \hbox{for some time} \ t \in (0, s) \}  \vspace*{3pt} \\
   \mathbf B & \n := \n & \{\zeta_t^2 (B) \leq 2^{-5} \,M^2 \ \hbox{for some time} \ t \in (0, s) \}. \end{array}
\end{equation}
 Applying Lemma~\ref{lem:keep} with~$K = aM$ then~$K = 2^{-3} \,M^2$, we get
\begin{equation}
\label{eq:move-3}
  P^* (\mathbf A) \leq \exp (- 2^{-4} \,aM) \quad \hbox{and} \quad P^* (\mathbf B) \leq \exp (- 2^{-7} \,M^2).
\end{equation}
 In addition, observing that the conditional payoff of each type~1 player in the set~$A$ given that the number of type~2 players in the set~$B$ is large satisfies
 $$ \phi (x \,| \,\xi_t (x) = 1, \mathbf B^c) \geq a_{12} \ 2^{-5} \,M^2 \quad \hbox{for all} \quad t \in (0, s) $$
 we deduce that, on the event~$(\mathbf A \cup \mathbf B)^c$, each time a player in the set~$B$ updates her strategy, she remains/becomes of type~1 with probability at least
\begin{equation}
\label{eq:move-4}
  \begin{array}{rcl}
    p_1 & \n \geq \n & (a_{12} \ 2^{-5} \,M^2)(2^{-2} \,aM)((a_{12} \ 2^{-5} \,M^2)(2^{-2} \,aM) + (a_{21} + a_{22})(2M + 1)^4)^{-1} \vspace*{4pt} \\
        & \n \geq \n & (a_{12} \ 2^{-5} \,M^2)(2^{-2} \,aM)(a_{21} \,2^5 \,M^4)^{-1} \geq 2^{-12} \,(a_{12} / a_{21}) \,aM^{-1} \vspace*{4pt} \\
        & \n \geq \n &  2^5 \,c_- (aM^{-1}) \end{array}
\end{equation}
 for all~$M$ large.
 In particular, letting
\begin{equation}
\label{eq:move-5}
  \begin{array}{rcl}
    X_u & \n := \n & \card \,\{x \in B : T_1 (x) < s \} \vspace*{3pt} \\
    X_1 & \n := \n & \card \,\{x \in B : T_1 (x) < s \ \hbox{and} \ \xi_s (x) = 1 \} \end{array}
\end{equation}
 observing from~\eqref{eq:move-4} that~$2^{-5} \,M^2 \,p_1 \geq c_- (aM)$, and using~\eqref{eq:keep-1}, we deduce that
\begin{equation}
\label{eq:move-6}
  \begin{array}{l}
    P \,(X_1 \leq c_- (aM) \ | \ (\mathbf A \cup \mathbf B)^c) \vspace*{4pt} \\ \hspace*{20pt} \leq \
    P \,(X_u \leq 2^{-4} \,M^2) + P \,(X_1 \leq c_- (aM) \ | \ (\mathbf A \cup \mathbf B)^c \ \hbox{and} \ X_u > 2^{-4} \,M^2) \vspace*{4pt} \\ \hspace*{20pt} \leq \
    P \,(\binomial (M^2 / 4, 1/2) \leq 2^{-4} \,M^2) + P \,(\binomial (2^{-4} \,M^2, p_1) \leq c_- (aM)) \vspace*{4pt} \\ \hspace*{20pt} \leq \
    P \,(\binomial (M^2 / 4, 1/2) \leq 2^{-4} \,M^2) + P \,(\binomial (2^{-4} \,M^2, p_1) \leq 2^{-5} \,M^2 \,p_1) \vspace*{4pt} \\ \hspace*{20pt} \leq \
      \exp (- 2^{-6} \,M^2) + \exp (- 2^{-7} \,M^2 \,p_1) \leq (1/2) \exp (- (aM)^{1/2}). \end{array}
\end{equation}
 Using that~$X_1 \leq \zeta_s^1 (B)$ and combining~\eqref{eq:move-3} and~\eqref{eq:move-6}, we conclude that
 $$ \begin{array}{l}
      P^* (\zeta_s^1 (B) \leq c_- (aM)) \leq P^* (\mathbf A \cup \mathbf B) + P \,(\zeta_s^1 (B) \leq c_- (aM) \ | \ (\mathbf A \cup \mathbf B)^c) \vspace*{4pt} \\ \hspace*{20pt}
                                        \leq \ P^* (\mathbf A) + P^* (\mathbf B) + P \,(X_1 \leq c_- (aM) \ | \ (\mathbf A \cup \mathbf B)^c) \vspace*{4pt} \\ \hspace*{20pt}
                                        \leq \ \exp (- 2^{-4} \,aM) + \exp (- 2^{-7} \,M^2) + (1/2) \exp (- (aM)^{1/2}) \leq \exp (- (aM)^{1/2}) \end{array} $$
 for all~$M$ large, which gives~\eqref{eq:move-1}.
 This completes the proof.
\end{proof} \\


\noindent{\bf Creating a pile of players} --
 The next lemma improves the previous one by showing that if the region~$A$ has a large number of type~1 players then the same amount of type~1 players can be created in the nearby region~$B$.
 The idea is to first prove that, as long as the number of type~1 players nearby is small, the number of such players can be increased by a factor~$c_+$.
 Once this threshold is reached, one can find a small box with a large number of type~1 players and apply the previous lemma repeatedly to move a fraction of these players to the target set~$B$.
\begin{lemma} --
\label{lem:create}
 Assume that~\eqref{eq:condition-coex} holds. Then,
 $$ P \,(\zeta_{6s}^1 (B) \leq M \ | \ \zeta_0^1 (A) \geq M) \leq \exp (- M^{1/2}) \quad \hbox{for all~$M$ large.} $$
\end{lemma}
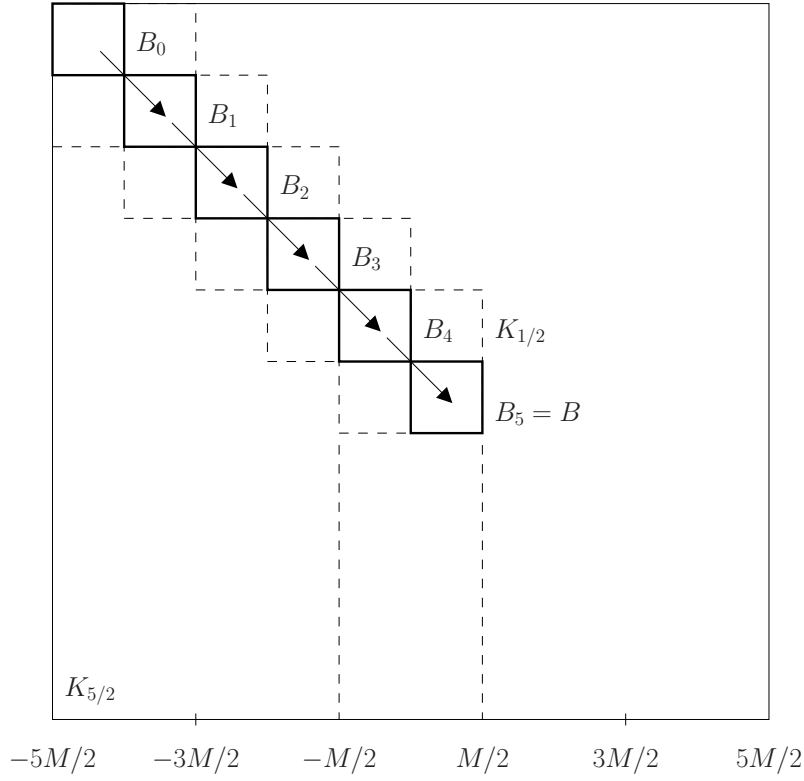
\begin{figure}[t]
\centering
\scalebox{0.50}{\input{coex.pstex_t}}
\caption{\upshape{Picture related to the proof of Lemma~\ref{lem:create}.}}
\label{fig:coex}
\end{figure}
\begin{proof}
 To keep track of the amount of type~1 and type~2 players in~$K_{5/2}$, which is the key to controlling the payoff of the type~1 players in the set~$A$, we set
 $$ \mathbf K := \{\zeta_t^1 (K_{5/2}) \geq c_+ M \ \hbox{for some time} \ t \in (0, s) \} $$
 where~$c_+$ has been defined in~\eqref{eq:constant-coex}.
 The proof is divided into two steps. \vspace*{5pt} \\
\noindent{\bf Step 1} -- First, we prove that the event~$\mathbf K$ occurs with probability close to one.
 The idea is to show that the complement of this event confers a large payoff to the type~1 players in~$A$, which results in a large
 production of such players with high probability, thus leading to a contradiction.
 To make this argument precise, we observe that, with probability one,
 $$ \zeta_t^2 (x + K_1) \geq M^2 \ \ \hbox{for all} \ x \in A \quad \hbox{and} \quad \zeta_t^1 (K_{5/2}) \leq c_+ M $$
 on the event~$\mathbf K^c$ and for all~$t \in (0, s)$.
 $$ \begin{array}{rcl}
      \phi (x \,| \,\xi_t (x) = 1, \mathbf K^c) & \n \geq \n & a_{12} \,M^2 \quad \hbox{for all} \quad x \in A \vspace*{2pt} \\
      \phi (x \,| \,\xi_t (x) = 2, \mathbf K^c) & \n \leq \n & a_{21} \,c_+ M + a_{22} \,(2M + 1)^2 \leq 2^3 \,a_{22} \,M^2 \quad \hbox{for all} \quad x \in K_{3/2}. \end{array} $$
 In particular, given~$(\mathbf A' \cup \mathbf K)^c$ where~$\mathbf A'$ is the first event~\eqref{eq:move-2} for~$a = 1$, each time a player in~$B$ updates her strategy, she remains/becomes of type~1 with probability at least
\begin{equation}
\label{eq:create-1}
  \begin{array}{rcl}
    q_1 & \n \geq \n & (a_{12} \,M^2)(2^{-2} \,M)((a_{12} \,M^2)(2^{-2} \,M) + (2^3 \,a_{22} \,M^2)(2M + 1)^2)^{-1} \vspace*{4pt} \\
        & \n \geq \n & (a_{12} \,M^2)(2^{-2} \,M)((2^4 \,a_{22} \,M^2)(2M + 1)^2)^{-1} \vspace*{4pt} \\
        & \n \geq \n & 2^{-9} \,(a_{12} / a_{22}) \,M^{-1}. \end{array}
\end{equation}
 Defining~$X_u$ and~$X_1$ as in~\eqref{eq:move-5}, observing that, by~\eqref{eq:condition-coex} and~\eqref{eq:create-1},
 $$ 2^{-5} \,M^2 \,q_1 \geq 2^{-14} \,(a_{12} / a_{22}) \,M \geq c_+ M $$
 and using~\eqref{eq:keep-1}, we deduce that, for all~$M$ sufficiently large,
\begin{equation}
\label{eq:create-2}
  \begin{array}{l}
    P \,(X_1 \leq c_+ M \ | \ (\mathbf A' \cup \mathbf K)^c) \vspace*{4pt} \\ \hspace*{20pt} \leq \
    P \,(X_u \leq 2^{-4} \,M^2) + P \,(X_1 \leq c_+ M \ | \ (\mathbf A' \cup \mathbf K)^c \ \hbox{and} \ X_u > 2^{-4} \,M^2) \vspace*{4pt} \\ \hspace*{20pt} \leq \
    P \,(\binomial (M^2 / 4, 1/2) \leq 2^{-4} \,M^2) + P \,(\binomial (2^{-4} \,M^2, q_1) \leq c_+ M) \vspace*{4pt} \\ \hspace*{20pt} \leq \
    P \,(\binomial (M^2 / 4, 1/2) \leq 2^{-4} \,M^2) + P \,(\binomial (2^{-4} \,M^2, q_1) \leq 2^{-5} \,M^2 \,q_1) \vspace*{4pt} \\ \hspace*{20pt} \leq \
      \exp (- 2^{-6} \,M^2) + \exp (- 2^{-7} \,M^2 \,q_1) \leq (1/4) \exp (- M^{1/2}). \end{array}
\end{equation}
 Using again~$X_1 \leq \zeta_s^1 (B)$ and combining~\eqref{eq:move-3} with~$a = 1$ and~\eqref{eq:create-2}, we get
\begin{equation}
\label{eq:create-3}
  \begin{array}{rcl}
    P \,(\mathbf K^c \ | \ \zeta_0^1 (A) \geq M) & \n = \n & P \,(\zeta_s^1 (B) \leq c_+ M \ \hbox{and} \ \mathbf K^c \ | \ \zeta_0^1 (A) \geq M) \vspace*{4pt} \\ \hspace*{25pt}
                                                 & \n \leq \n & P \,(\mathbf A' \ | \ \zeta_0^1 (A) \geq M) + P \,(\zeta_s^1 (B) \leq c_+ M \ | \ (\mathbf A' \cup \mathbf K)^c) \vspace*{4pt} \\ \hspace*{25pt}
                                                 & \n \leq \n & \exp (- 2^{-4} \,M) + (1/4) \exp (- M^{1/2}) \vspace*{4pt} \\ \hspace*{25pt}
                                                 & \n \leq \n & (1/2) \exp (- M^{1/2}) \end{array}
\end{equation}
 for all dispersal range~$M$ sufficiently large. \vspace*{5pt} \\
\noindent{\bf Step 2} -- Now, given~$\mathbf K$, there is a box with at least~$5^{-2} \,c_+ M$ type~1 players.
 One can move a fraction of these players to the target set~$B$ in at most five steps, applying Lemma~\ref{lem:move}.
 To begin with, we observe that some basic geometry implies that, given~$\mathbf K$, there exist
 $$ B_0, \ldots, B_5 \subset K_{5/2} \quad \hbox{and} \quad t_0 \in (0, s) $$
 such that the following three conditions hold:
\begin{itemize}
 \item[(a)] We have~$\zeta_{t_0}^1 (B_0) \geq 5^{-2} \,c_+ M$. \vspace*{4pt}
 \item[(b)] For~$k = 0, \ldots, 5$, we have~$\card (B_k) = 2^{-2} \,M^2$ with~$B_5 = B$. \vspace*{4pt}
 \item[(c)] For~$k = 0, \ldots, 4$, we have~$\max_{j = 1, 2} |x_j - y_j| \leq M$ for all~$(x, y) \in B_k \times B_{k + 1}$.
\end{itemize}
 We refer to Figure~\ref{fig:coex} for an illustration of the worst case scenario where all the type~1 players are located in one of the corners of~$K_{5/2}$.
 Under these conditions, we can bring type~1 players to our target set in at most five steps applying repeatedly Lemma~\ref{lem:move}. Indeed,
\begin{equation}
\label{eq:create-4}
  \begin{array}{l}
    P \,(\zeta_t^1 (B) \leq 2^7 M \ \hbox{for all} \ t \in (0, 6s) \ | \ \mathbf K) \vspace*{4pt} \\ \hspace*{25pt} \leq \
    P \,(\zeta_{t_0 + 5s}^1 (B_5) \leq 2^7 M \ | \ \zeta_{t_0}^1 (B_0) \geq 5^{-2} \,c_+ M = 2^7 (c_-)^{-5} \,M) \vspace*{4pt} \\ \hspace*{25pt} \leq \
    1 - \prod_{k = 0, 1, 2, 3, 4} \,P \,(\zeta_s^1 (B_{k + 1}) \geq 2^7 (c_-)^{k - 4} \,M \ | \ \zeta_0^1 (B_k) \geq 2^7 (c_-)^{k - 5} \,M) \vspace*{4pt} \\ \hspace*{25pt} \leq \
    1 - \prod_{k = 0, 1, 2, 3, 4} \ (1 - \exp (- 2^{7/2} \,(c_-)^{(k - 5) / 2} \,M^{1/2})) \vspace*{4pt} \\ \hspace*{25pt} \leq \
    1 - (1 - \exp (- 2^{7/2} \,M^{1/2}))^5 \vspace*{4pt} \\ \hspace*{25pt} \leq \
    5 \times \exp (- 2^{7/2} \,M^{1/2}) \leq (1/4) \exp (- M^{1/2}) \end{array}
\end{equation}
 for all dispersal range~$M$ sufficiently large. \vspace*{5pt} \\
\noindent{\bf Conclusion} --
 Combining~\eqref{eq:create-3}--\eqref{eq:create-4}, we deduce that
 $$ \begin{array}{l}
      P \,(\zeta_t^1 (B) \leq 2^7 M \ \hbox{for all} \ t \in (0, 6s) \ | \ \zeta_0^1 (A) \geq M) \vspace*{4pt} \\ \hspace*{25pt} \leq \
      P \,(\mathbf K^c \ | \ \zeta_0^1 (A) \geq M) + P \,(\zeta_t^1 (B) \leq 2^7 M \ \hbox{for all} \ t \in (0, 6s) \ | \ \mathbf K) \vspace*{4pt} \\ \hspace*{25pt} \leq \
      (1/2) \exp (- M^{1/2}) + (1/4) \exp (- M^{1/2}) = (3/4) \exp (- M^{1/2}) \end{array} $$
 which, applying Lemma~\ref{lem:keep} with~$n = 6$ and~$K = 2^7 M$, implies that
 $$ \begin{array}{l}
      P \,(\zeta_{6s}^1 (B) \leq M \ | \ \zeta_0^1 (A) \geq M) \vspace*{4pt} \\ \hspace*{25pt} \leq \
      P \,(\zeta_t^1 (B) \leq 2^7 M \ \hbox{for all} \ t \in (0, 6s) \ | \ \zeta_0^1 (A) \geq M) \vspace*{4pt} \\ \hspace*{60pt} + \
      P \,(\zeta_{6s}^1 (B) \leq 2^{-7} \,2^7 M \ | \ \zeta_t^1 (B) \geq 2^7 M \ \hbox{for some} \ t \in (0, 6s)) \vspace*{4pt} \\ \hspace*{25pt} \leq \
      (3/4) \exp (- M^{1/2}) + \exp (- 2^{-2} \,M) \leq \exp (- M^{1/2}) \end{array} $$
 for all~$M$ large.
 This completes the proof.
\end{proof} \\


\noindent{\bf Block construction} --
 To deduce coexistence, we use Lemma~\ref{lem:create} in combination with some obvious symmetry and a block construction.
 The idea of the block construction is to define a coupling between the process under consideration properly rescaled in space and time and supercritical percolation.
 More precisely, let~$\H$ be the directed graph with vertex set
 $$ H := \{(z, n) \in \Z^2 \times \Z_+ : z_1 + z_2 + n \ \hbox{is even} \} $$
 and in which there is an oriented edge
 $$ \begin{array}{l}
     (z, n) \to (z', n') \quad \hbox{if and only if} \quad (z' = z \pm e_1 \ \hbox{or} \ z' = z \pm e_2) \ \ \hbox{and} \ \ n' = n + 1 \end{array} $$
 where $e_j$ is the $j$th unit vector.
 Then, we consider the 14 dependent oriented site percolation process with density equal to~$1 - \ep$ on this directed graph, i.e., we assume that
 $$ P \,((z_i, n_i) \ \hbox{is closed for} \ i = 1, 2, \ldots, m) = \ep^m $$
 whenever~$|z_i - z_j| \vee |n_i - n_j| > 14$ for~$i \neq j$.
 Then, we set
 $$ B_z := (M / 2) \,z + [- M/4, M/4)^2 \ \ \hbox{for all} \ \ z \in \Z^2 \quad \hbox{and} \quad T := 6s = 6 \,\ln (2) $$
 and declare site~$(z, n) \in H$ to be good whenever
 $$ \zeta_{nT}^i (B_z) = \card \,\{x \in B_z : \xi_{nT} (x) = i \} \geq M \quad \hbox{for} \quad i = 1, 2. $$
 Finally, for all~$n \in \N$, we define
 $$ W_n^{\ep} := \{z : (z, n) \ \hbox{is wet} \} \quad \hbox{and} \quad X_n := \{z : (z, n) \ \hbox{is good} \}. $$
 The next lemma shows that, for all~$\ep > 0$, one can find a sufficiently large dispersal range such that the set of good sites dominates stochastically the set of wet sites.
 In view of the definition of a good site, this will imply coexistence of both strategies.
\begin{lemma} --
\label{lem:coex-perco}
 Assume~\eqref{eq:condition-coex} and fix~$\ep > 0$.
 Then, for all~$M$ large,
  $$ P \,(z \in W_n^{\ep}) \leq P \,(z \in X_n) \quad \hbox{for all} \quad (z, n) \in H \quad \hbox{whenever} \quad W_0^{\ep} \subset X_0. $$
\end{lemma}
\begin{proof}
 First, we define the collection of events
 $$ \mathbf B_i (z, n) := \{\zeta_{nT}^i (B_z) \geq M \} \quad \hbox{for all} \quad (z, n) \in H \ \hbox{and} \ i = 1, 2. $$
 Then, since for~$j = 1, 2$,
 $$ B_z, B_{z \pm e_j} \subset ((M/2) \,z \pm (M/4) \,e_j) + K_{1/2} \quad \hbox{and} \quad \card (B_z) = \card (B_{z \pm e_j}) = 2^{-2} \,M^2 $$
 we can apply Lemma~\ref{lem:create} to get
 $$ \begin{array}{l}
      P \,(\mathbf B_1 (z \pm e_j, n + 1) \ \hbox{for} \ j = 1, 2 \ | \ \mathbf B_1 (z, n)) \vspace*{4pt} \\ \hspace*{25pt} = \
      P \,(\zeta_{(n + 1) T}^1 (B_{z \pm e_j}) \geq M \ \hbox{for} \ j = 1, 2 \ | \ \zeta_{nT}^1 (B_z) \geq M) \vspace*{4pt} \\ \hspace*{25pt} \geq \
      1 - 4 \times P \,(\zeta_{(n + 1) T}^1 (B_{z + e_1}) \geq M \ | \ \zeta_{nT}^1 (B_z) \geq M) \geq 1 - 4 \times \exp (- M^{1/2}) \vspace*{4pt} \\ \hspace*{25pt} \geq \
      1 - \ep/2 \end{array} $$
 for all~$M$ large.
 By symmetry, the same holds for strategy~2, therefore
\begin{equation}
\label{eq:coex-perco-1}
\begin{array}{l}
   P \,(z \pm e_j \in X_{n + 1} \ \hbox{for} \ j = 1, 2 \ | \ z \in X_n) \vspace*{4pt} \\ \hspace*{25pt} = \
   P \,(\mathbf B_1 (z \pm e_j,  n + 1) \cap \mathbf B_2 (z \pm e_j, n + 1) \ \hbox{for} \ j = 1, 2 \ | \ \mathbf B_1 (z, n) \cap \mathbf B_2 (z, n)) \vspace*{4pt} \\ \hspace*{25pt} \geq \
   - \ 1 + P \,(\mathbf B_1 (z \pm e_j,  n + 1) \ \hbox{for} \ j = 1, 2 \ | \ \mathbf B_1 (z, n) \cap \mathbf B_2 (z, n)) \vspace*{4pt} \\ \hspace*{100pt} + \
           P \,(\mathbf B_2 (z \pm e_j,  n + 1) \ \hbox{for} \ j = 1, 2 \ | \ \mathbf B_1 (z, n) \cap \mathbf B_2 (z, n)) \vspace*{4pt} \\ \hspace*{25pt} \geq \
   - \ 1 + 2 \,(1 - \ep/2) = 1 - \ep. \end{array}
\end{equation}
 Since in addition all the estimates in Lemma~\ref{lem:create} hold regardless of the configuration outside the spatial region~$K_{5/2}$, we deduce from~\eqref{eq:coex-perco-1} the existence of a collection of events
 $$ \mathbf G (z, n) \quad \hbox{for all} \quad (z, n) \in H $$
 that satisfy the following three properties:
\begin{itemize}
 \item[(a)] The event~$\mathbf G (z, n)$ is measurable with respect to the graphical representation in
  $$ R (z, n) := ((M/2) z, nT) + (K_{7/2} \times [0, T]). $$
 \item[(b)] For all~$M$ large, we have~$P \,(\mathbf G (z, n)) \geq 1 - \ep$. \vspace*{4pt}
 \item[(c)] We have the inclusion~$\mathbf G (z, n) \cap \{z \in X_n \} \subset \{z \pm e_j \in X_{n + 1} \ \hbox{for} \ j = 1, 2 \}$.
\end{itemize}
 Observing also that
 $$ R (z, n) \cap R (z', n') = \varnothing \quad \hbox{when} \quad |z - z'| \vee |n - n'| \geq 2 \times 7 = 14, $$
 we deduce from~\cite[Theorem 4.3]{durrett_1995} the existence of a coupling between the long range death-birth process and the oriented site percolation process such that
 $$ P \,(W_n^{\ep} \subset X_n) = 1 \quad \hbox{whenever} \quad W_0^{\ep} \subset X_0. $$
 The lemma directly follows from the existence of this coupling.
\end{proof} \\ \\
 To conclude, we fix~$\ep > 0$ small enough to make the percolation process supercritical.
 Observing that, starting from a product measure with a positive density of both strategies, the number of good sites at level zero is almost surely infinite, we deduce that
 $$ \begin{array}{l} \liminf_{n \to \infty} \,P \,(0 \in W_{2n} \ | \ W_0 = X_0) > 0. \end{array} $$
 This, together with Lemma~\ref{lem:coex-perco}, implies that, for all~$M$ large,
 $$ \begin{array}{l}
    \liminf_{t \to \infty} \,P \,(\xi_t (x) = 1) \,P \,(\xi_t (x) = 2) \vspace*{4pt} \\ \hspace*{20pt} = \
    \liminf_{t \to \infty} \,P \,(\xi_t (0) = 1) \,P \,(\xi_t (0) = 2) \vspace*{4pt} \\ \hspace*{20pt} \geq \
    \liminf_{t \to \infty} \,\prod_{i = 1, 2} \,P \,(\xi_t (0) = i \ | \ 0 \in W_{2 \floor{t/T}}) \,P \,(0 \in W_{2 \floor{t/T}} \ | \ W_0 = X_0) > 0 \end{array} $$
 for all~$x \in \Z^2$.
 This completes the proof of Theorem~\ref{th:coex}.


\section{Coupling with modified voter models}
\label{sec:coupling}

\indent This section is devoted to the proof of Theorem~\ref{th:coupling}.
 The common ingredient to prove both parts of the theorem is to couple the process with the modified voter models~$\zeta_t^1$ and~$\zeta_t^2$ whose transitions at vertex~$x$ are given by the following expressions
 $$ \begin{array}{rcl}
     i \ \to \ j & \hbox{at rate} & c_{i \to j} (x, \zeta^i) := (1 - \ep) \,f_j (x, \zeta^i) + \ep \ \ind \{f_i (x, \zeta^i) = 0 \} \vspace*{4pt} \\
     j \ \to \ i & \hbox{at rate} & c_{j \to i} (x, \zeta^i) := (1 - \ep) \,f_i (x, \zeta^i) + \ep \ \ind \{f_i (x, \zeta^i) \neq 0 \}. \end{array} $$
 for~$\{i, j \} = \{1, 2 \}$ and where
 $$ f_j (x, \zeta^i) = \card \,\{y \in N_x : \zeta^i (x) = j \} / \card N_x = (1/N) \,N_j (x, \zeta^i) $$
 denotes the fraction of neighbors of vertex~$x$ in state~$j$.
 In words, the transition rates indicate that particles are updated at rate one and that, at the time of an update,
\begin{itemize}
 \item with probability $1 - \ep > 0$, the new type is chosen uniformly at random from the interaction neighborhood just like in the voter model, \vspace{4pt}
 \item with probability $\ep > 0$, the new type is~$i$ unless all the neighbors are of type~$j$.
\end{itemize}
 The results of~\cite[section~3]{lanchier_2013} show using duality that type~$i$ particles win for this process.
 In particular, to prove that strategy~1 wins, it suffices to prove that the set of type~1 players in the death-birth process dominates stochastically its counterpart in~$\zeta_t^1$, which follows from
 $$ \begin{array}{l}
    \xi \leq \zeta^1 \ \ \hbox{and} \ \ \xi (x) = \zeta^1 (x) \quad \hbox{implies that} \vspace*{4pt} \\ \hspace*{90pt}
     p_{1 \to 2} (x, \xi) \leq c_{1 \to 2} (x, \zeta^1) \ \ \hbox{and} \ \ p_{2 \to 1} (x, \xi) \geq c_{2 \to 1} (x, \zeta^1) \end{array} $$
 according to Theorem~III.1.5 in~\cite{liggett_1985}.
 Since in addition the transition rates of the modified voter models are monotone with respect to the number of neighbors of each type, in order to show that strategy~1 wins, it suffices to prove that the simplified implication
\begin{equation}
\label{eq:voter-1}
  \begin{array}{l}
    N_1 (x, \xi) = N_1 (x, \zeta^1) \ \ \hbox{and} \ \ \xi (x) = \zeta^1 (x) \quad \hbox{implies that} \vspace*{4pt} \\ \hspace*{90pt}
      p_{1 \to 2} (x, \xi) \leq c_{1 \to 2} (x, \zeta^1) \ \ \hbox{and} \ \ p_{2 \to 1} (x, \xi) \geq c_{2 \to 1} (x, \zeta^1) \end{array}
\end{equation}
 holds for some~$\ep > 0$.
 By symmetry, strategy~2 wins if the implication
\begin{equation}
\label{eq:voter-2}
  \begin{array}{l}
    N_2 (x, \xi) = N_2 (x, \zeta^2) \ \ \hbox{and} \ \ \xi (x) = \zeta^2 (x) \quad \hbox{implies that} \vspace*{4pt} \\ \hspace*{90pt}
      p_{1 \to 2} (x, \xi) \geq c_{1 \to 2} (x, \zeta^2) \ \ \hbox{and} \ \ p_{2 \to 1} (x, \xi) \leq c_{2 \to 1} (x, \zeta^2) \end{array}
\end{equation}
 holds for some~$\ep > 0$.
 Using~\eqref{eq:voter-1}, we now prove Theorem~\ref{th:coupling}.a.
\begin{lemma} --
\label{lem:1-wins}
 Recall that~$a_{21} > a_{12}$.
 Then, strategy~1 wins when
 $$ \min \,(a_{12} - a_{22}, a_{11} - a_{21}) > (N - 1)(a_{21} - a_{12}). $$
\end{lemma}
\begin{proof}
 In view of the discussion above, it suffices to prove that~\eqref{eq:voter-1} holds.
 First, we observe that, when the fraction of type~1 neighbors of vertex~$x$ is equal to either zero or one, the transition rates are the same for both processes so it remains to prove~\eqref{eq:voter-1} under the assumption
\begin{equation}
\label{eq:1-wins-1}
  N_1 N_2 \neq 0 \quad \hbox{where} \quad N_j := N_j (x, \xi) = N_j (x, \zeta^1) \quad \hbox{for} \quad j = 1, 2.
\end{equation}
 The transition rate at vertex~$x$ depends on the payoff of its neighbors, and the main idea is to express the transition rates by distinguishing between the part of the payoff coming from~$x$ and
 the part of the payoff coming from the other neighbors' neighbors.
 In order to make this distinction, we introduce the following four weighting factors:
 $$ \begin{array}{l} w_{ij} := \sum_{y \sim x} \,(\ind \{\xi (y) = i \} \,\sum_{z \sim y, z \neq x} \,\ind \{\xi (z) = j \}) \quad \hbox{for} \quad i, j = 1, 2. \end{array} $$
 That is, $w_{ij}$ is the number of type~$j$ neighbors (excluding vertex~$x$) of a type~$i$ neighbor of vertex~$x$ counted with order of multiplicity.
 Note that, for~$i = 1, 2$, we have
\begin{equation}
\label{eq:1-wins-2}
  \begin{array}{rcl}
    N_i + \sum_{j = 1, 2} \,w_{ij} & \n = \n & N_i + \sum_{y \sim x} \,(\ind \{\xi (y) = i \} \,\sum_{z \sim y} \,\ind \{z \neq x \}) \vspace*{4pt} \\
                                   & \n = \n & N_i + (N - 1) \,\sum_{y \sim x} \,\ind \{\xi (y) = i \} \vspace*{4pt} \\
                                   & \n = \n & N_i + (N - 1) \,N_i = NN_i. \end{array}
\end{equation}
 In addition, for~$i \neq j$, the transition rates can be expressed as
\begin{equation}
\label{eq:1-wins-3}
  p_{i \to j} (x, \xi) = \frac{(N_j + w_{ji}) \,a_{ji} + w_{jj} \,a_{jj}}{(N_i + w_{ii}) \,a_{ii} + w_{ij} \,a_{ij} + (N_j + w_{ji}) \,a_{ji} + w_{jj} \,a_{jj}}.
\end{equation}
 Using~\eqref{eq:1-wins-2}--\eqref{eq:1-wins-3}, we now prove that~\eqref{eq:voter-1} holds in the nontrivial case~\eqref{eq:1-wins-1}. \vspace*{5pt} \\
\noindent{\bf Transition~$1 \to 2$} --
 Using~\eqref{eq:1-wins-2} and~$a_{22} < a_{12} < a_{21}$, we get
 $$ \begin{array}{rcl}
     (N_2 + w_{21}) \,a_{21} + w_{22} \,a_{22} & \n \leq \n & (N_2 + w_{21} + w_{22}) \,a_{21} = NN_2 \,a_{21} \vspace*{4pt} \\
     (N_1 + w_{11}) \,a_{11} + w_{12} \,a_{12} & \n = \n & (N_1 + w_{11} + w_{12}) \,a_{11} + w_{12} \,(a_{12} - a_{11}) \vspace*{4pt} \\
                                               & \n = \n & NN_1 \,a_{11} + w_{12} \,(a_{12} - a_{11}). \end{array} $$
 This, together with~\eqref{eq:1-wins-3} for~$i = 1$ and~$j = 2$, implies that
\begin{equation}
\label{eq:1-wins-4}
  \begin{array}{rcl}
    p_{1 \to 2} (x, \xi) & \n \leq \n & \displaystyle \frac{NN_2 \,a_{21}}{NN_1 \,a_{11} + w_{12} \,(a_{12} - a_{11}) + NN_2 \,a_{21}} \vspace*{10pt} \\
                         & \n    = \n & \displaystyle \frac{NN_2 \,a_{21}}{NN_1 \,(a_{11} - a_{21}) + w_{12} \,(a_{12} - a_{11}) + N^2 \,a_{21}} = \frac{NN_2 \,a_{21}}{N^2 a_{21} + \rho_1} \end{array}
\end{equation}
 where, since~$a_{11} - a_{21} > (N - 1)(a_{21} - a_{12})$,
\begin{equation}
\label{eq:1-wins-5}
  \begin{array}{rrl}
    \rho_1 & \n := \n & NN_1 \,(a_{11} - a_{21}) + w_{12} \,(a_{12} - a_{11}) \vspace*{4pt} \\
           & \n  = \n & (NN_1 - w_{12})(a_{11} - a_{21}) + w_{12} \,(a_{12} - a_{21}) \vspace*{4pt} \\
           & \n  > \n & (N - 1)(NN_1 - w_{12})(a_{21} - a_{12}) + w_{12} \,(a_{12} - a_{21}) \vspace*{4pt} \\
           & \n  = \n & N \,((N - 1) \,N_1 - w_{12})(a_{21} - a_{12}) = N \,w_{11} \,(a_{21} - a_{12}) \geq 0. \end{array}
\end{equation}
 Note that the strict inequality holds because~\eqref{eq:1-wins-1}--\eqref{eq:1-wins-2} imply that
 $$ NN_1 - w_{12} = (N_1 + w_{11} + w_{12}) - w_{12} = N_1 + w_{11} \geq N_1 > 0. $$
 In view of~\eqref{eq:1-wins-4}--\eqref{eq:1-wins-5}, there exists~$\ep_1 > 0$ small such that
 $$ \begin{array}{rcl}
      p_{1 \to 2} (x, \xi) & \n \leq \n & NN_2 \,a_{21} \,(N^2 a_{21} + \rho_1)^{-1} = N_2 \,(N + (\rho_1 / Na_{21}))^{-1} \vspace*{4pt} \\
                           & \n \leq \n & (1 - \ep) \,f_2 (x, \xi) = (1 - \ep) \,f_2 (x, \zeta^1) = c_{1 \to 2} (x, \zeta^1) \end{array} $$
 whenever~$\ep < \ep_1$ and~\eqref{eq:1-wins-1} holds. \vspace*{5pt} \\
\noindent{\bf Transition~$2 \to 1$} --
 Since~$a_{11} > a_{21} > a_{12}$ and~$a_{22} - a_{12} < (N - 1)(a_{12} - a_{21})$, using the previous estimates and obvious symmetry, we show that
 $$ p_{2 \to 1} (x, \xi) \geq NN_1 \,a_{12} \,(N^2 a_{12} + \rho_2)^{-1} \quad \hbox{where} \quad \rho_2 < N \,w_{22} \,(a_{12} - a_{21}) \leq 0. $$
 In particular, there exists~$\ep_2 > 0$ small such that
 $$ \begin{array}{rcl}
      p_{2 \to 1} (x, \xi) & \n \geq \n & N_1 \,(N + (\rho_2 / Na_{12}))^{-1} \vspace*{4pt} \\
                           & \n \geq \n & (1 - \ep) \,f_1 (x, \xi) + \ep = (1 - \ep) \,f_1 (x, \zeta^1) + \ep = c_{2 \to 1} (x, \zeta^1) \end{array} $$
 whenever~$\ep < \ep_2$ and~\eqref{eq:1-wins-1} holds. \vspace*{5pt} \\
 In conclusion, the implication~\eqref{eq:voter-1} holds for~$\ep$ smaller than~$\min \,(\ep_1, \ep_2) > 0$, which shows that strategy~1 wins under the assumptions of the lemma.
\end{proof} \\ \\
 Repeating the proof of Lemma~\ref{lem:1-wins} step by step but exchanging the role of both strategies only shows that strategy~2 wins under the strong assumption~$a_{11} < a_{12} < a_{21} < a_{22}$.
 In fact, this sufficient condition for strategy~2 to win can be easily improved to
 $$ \max \,(a_{11}, a_{12}) < \min \,(a_{21}, a_{22}) $$
 by using a coupling between the death-birth process and a biased voter model with a selective advantage for type~2 particles to show that the former dominates the latter.
 To prove that strategy~2 wins in the larger region stated in Theorem~\ref{th:coupling}.b, we couple the death-birth process with the second modified voter model~$\zeta_t^2$ but using techniques different from the ones used
 to show the first part of the theorem.
 To explain the assumptions of the theorem, we note that our approach requires the interaction neighborhood to have a certain connectivity property which does not hold in the one-dimensional nearest neighbor case.
 First, we introduce the payoff functions
 $$ \begin{array}{rcl}
    \phi_1 (z) & \n := \n & a_{11} \,(z/N) + a_{12} \,(1 - z/N) = (a_{11} - a_{12})(z/N) + a_{12} \vspace*{3pt} \\
    \phi_2 (z) & \n := \n & a_{22} \,(z/N) + a_{21} \,(1 - z/N) = (a_{22} - a_{21})(z/N) + a_{21} \end{array} $$
 for all~$z = 0, 1, \ldots, N$, and let~$a := \max \,(a_{11}, a_{12}, a_{21}, a_{22})$ and
 $$ \begin{array}{rcl}
      M_+ := \max_z \,\phi_1 (z) & \n \geq \n & \max_{z \neq N} \,\phi_1 (z) =: M_-  \vspace*{4pt} \\
      m_- := \min_z \,\phi_2 (z) & \n \leq \n & \min_{z \neq N} \,\phi_2 (z) =: m_+. \end{array} $$
 The following lemma gives a sufficient condition on these minimum and maximum payoffs for strategy~2 to win.
 This condition is made more explicit in the subsequent lemma.
\begin{lemma} --
\label{lem:2-wins}
 Strategy~2 wins whenever~$(M, d) \neq (1, 1)$ and
\begin{equation}
\label{eq:2-wins-0}
  (N - 1) \,m_+ > (N - 2) \,M_+ + M_- \quad \hbox{and} \quad (N - 1) \,M_- < (N - 2) \,m_- + m_+.
\end{equation}
\end{lemma}
\begin{proof}
 Following as in the proof of Lemma~\ref{lem:1-wins}, it suffices to show that~\eqref{eq:voter-2} holds.
 This is again trivial when the fraction of type~1 neighbors of vertex~$x$ is equal to either zero or one so we focus from now on on the nontrivial case where
\begin{equation}
\label{eq:2-wins-1}
  N_1 N_2 \neq 0 \quad \hbox{where} \quad N_j := N_j (x, \xi) = N_j (x, \zeta^2) \quad \hbox{for} \quad j = 1, 2,
\end{equation}
 indicating that~$x$ has two neighbors~$y^*$ and~$z^*$ with different strategies.
 Except in the one-dimensional nearest neighbor case~$M = d = 1$, one can find two vertices~$y_*$ and~$z_*$ such that
 $$ y_*, z_* \neq x \quad \hbox{and} \quad y_* \in N_x \cap N_{y^*} \quad \hbox{and} \quad z_* \in N_x \cap N_{z^*} \quad \hbox{and} \quad y_* \in N_{z_*} $$
 and we may assume that~$y^*$ and~$z^*$ are neighbors of each other:
 $$ \xi (y^*) = 1 \quad \hbox{and} \quad \xi (z^*) = 2 \quad \hbox{and} \quad y^*, z^* \in N_x \quad \hbox{and} \quad y^* \in N_{z^*}. $$
 The rest of the proof is divided into two steps depending on the transition. \vspace*{5pt} \\
\noindent{\bf Transition~$1 \to 2$} --
 Assume that~$\xi (x) = \zeta^2 (x) = 1$. Then,
 $$ \phi (y \,| \,\xi (y) = 1) \leq M_+ \quad \hbox{and} \quad \phi (y \,| \,\xi (y) = 2) \geq m_+ \quad \hbox{for all} \quad y \in N_x. $$
 In addition, we have~$\phi (y^*, \xi) \leq M_-$ therefore
\begin{equation}
\label{eq:2-wins-2}
  \begin{array}{rcl}
    p_{1 \to 2} (x, \xi) & \n \geq \n & \displaystyle \frac{m_+ \,N_2}{(N_1 - 1) \,M_+ + M_- + m_+ \,N_2} \vspace*{8pt} \\
                         & \n   =  \n & \displaystyle \frac{m_+ \,N_2}{(m_+ - M_+) \,N_2 + (N - 1) \,M_+ + M_-}. \end{array}
\end{equation}
 Now, for all~$z \in [0, N]$, we define the functions
 $$ g_1 (z) := \frac{m_+ \,z}{(m_+ - M_+) \,z + (N - 1) \,M_+ + M_-} \quad \hbox{and} \quad h_1 (z) := \frac{(1 - \ep) \,z}{N} + \ep. $$
 Since~$(N - 1) \,m_+ > (N - 2) \,M_+ + M_-$,
\begin{equation}
\label{eq:2-wins-3}
  g_1 (1) = \frac{m_+}{m_+ + (N - 2) \,M_+ + M_-} > \frac{1}{N}.
\end{equation}
 We then distinguish two cases. \vspace*{5pt} \\
\noindent{\bf Case 1} -- When~$M_+ > m_+$, it follows from~\eqref{eq:2-wins-3} that
 $$ g_1 (z) \geq g_1 (1) \,z > z/N \quad \hbox{for all} \quad z \in [1, N - 1] $$
 therefore~$g_1 (z) \geq h_1 (z)$ for all~$\ep > 0$ small by continuity. \vspace*{5pt} \\
\noindent{\bf Case 2} -- When~$M_+ \leq m_+$, the function~$g_1$ is concave down. Moreover,
 $$ g_1 (1) \geq \frac{1}{N} + \bigg(1 - \frac{1}{N} \bigg) \,\ep = h_1 (1) $$
 for all~$\ep > 0$ small according to~\eqref{eq:2-wins-3}, while
 $$ g_1 (N) = \frac{N m_+}{N m_+ - M_+ + M_-} \geq 1 = h_1 (N). $$
 This again implies that~$g_1$ dominates~$h_1$ for all~$\ep > 0$ small (see Figure~\ref{fig:concave} for a picture). \vspace*{5pt} \\
 Recalling~\eqref{eq:2-wins-2}, we deduce that, in both cases and when~\eqref{eq:2-wins-1} holds,
 $$ p_{1 \to 2} (x, \xi) \geq g_1 (N_2) \geq h_1 (N_2) = c_{1 \to 2} (x, \zeta^2) $$
 which proves the first inequality in~\eqref{eq:voter-2}. \vspace*{5pt} \\
\noindent{\bf Transition~$2 \to 1$} --
 Assume that~$\xi (x) = \zeta^2 (x) = 2$. Then,
 $$ \phi (y \,| \,\xi (y) = 1) \leq M_- \quad \hbox{and} \quad \phi (y \,| \,\xi (y) = 2) \geq m_- \quad \hbox{for all} \quad y \in N_x. $$
 In addition, we have~$\phi (z^*, \xi) \geq m_+$ therefore
\begin{equation}
\label{eq:2-wins-4}
  \begin{array}{rcl}
    p_{2 \to 1} (x, \xi) & \n \leq \n & \displaystyle \frac{M_- \,N_1}{M_- \,N_1 + (N_2 - 1) \,m_- + m_+} \vspace*{8pt} \\
                         & \n   =  \n & \displaystyle \frac{M_- \,N_1}{(M_- - m_-) \,N_1 + (N - 1) \,m_- + m_+}. \end{array}
\end{equation}
 Now, for all~$z \in [0, N]$, we define the functions
 $$ g_2 (z) := \frac{M_- \,z}{(M_- - m_-) \,z + (N - 1) \,m_- + m_+} \quad \hbox{and} \quad h_2 (z) := \frac{(1 - \ep) \,z}{N}. $$
 Since~$(N - 1) \,M_- < (N - 2) \,m_- + m_+ \leq (N - 1) \,m_+$,
\begin{equation}
\label{eq:2-wins-5}
  \begin{array}{rcl}
        g_2 (1) & \n = \n & \displaystyle \frac{M_-}{M_- + (N - 2) \,m_- + m_+} < \frac{1}{N} \vspace*{8pt} \\
    g_2 (N - 1) & \n = \n & \displaystyle \frac{(N - 1) \,M_-}{(N - 1) \,M_- + m_+} < 1 - \frac{1}{N}. \end{array}
\end{equation}
 As before, we distinguish two cases. \vspace*{5pt} \\
\noindent{\bf Case 1} -- When~$M_- > m_-$, it follows from~\eqref{eq:2-wins-5} that
 $$ g_2 (z) \leq g_2 (1) \,z < z/N \quad \hbox{for all} \quad z \in [1, N - 1] $$
 therefore~$g_2 (z) \leq h_2 (z)$ for all~$\ep > 0$ small by continuity. \vspace*{5pt} \\
\noindent{\bf Case 2} -- When~$M_- \leq m_-$, the function~$g_2$ is concave up. Moreover,
 $$ \begin{array}{rcl}
      g_2 (N - 1) & \n \geq \n & (1 - 1/N)(1 - \ep) = h_2 (N - 1) \ \ \hbox{for} \ \ \ep > 0 \ \hbox{small} \vspace*{4pt} \\
          g_2 (0) & \n   =  \n & h_2 (0) = 0 \end{array} $$
 where the first inequality follows from~\eqref{eq:2-wins-5}.
 This again implies that~$h_2$ dominates~$g_2$ for all~$\ep > 0$ small, and we refer to the right-hand side of Figure~\ref{fig:concave} for a picture. \vspace*{5pt} \\
 Recalling~\eqref{eq:2-wins-4}, we deduce that, in both cases and when~\eqref{eq:2-wins-1} holds,
 $$ p_{2 \to 1} (x, \xi) \leq g_2 (N_1) \leq h_2 (N_1) = c_{2 \to 1} (x, \zeta^2) $$
 which proves the second inequality in~\eqref{eq:voter-2}. \vspace*{5pt} \\
 This completes the proof.
\begin{figure}[t]
\centering
\scalebox{0.50}{\input{concave.pstex_t}}
\caption{\upshape{Picture related to the proof of Lemma~\ref{lem:2-wins}.}}
\label{fig:concave}
\end{figure}
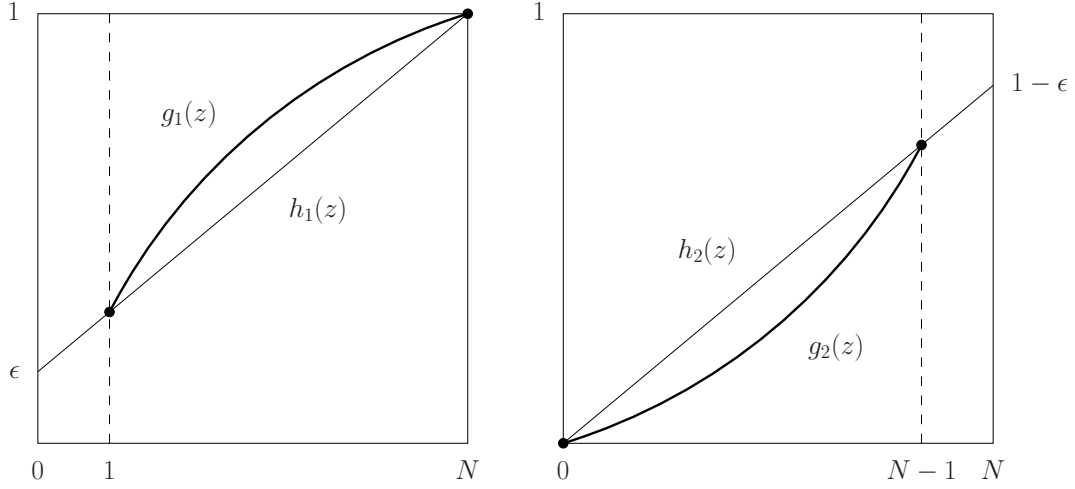
\end{proof} \\ \\
 To complete the proof of Theorem~\ref{th:coupling}, the last step is to re-express the condition in the previous lemma using the payoff coefficients.
\begin{lemma} --
\label{lem:region}
 Assume that~$a_{12} < a_{21}$.
 Then~\eqref{eq:2-wins-0} holds whenever
\begin{equation}
\label{eq:region-0}
  (N^2 - N - 1) \,\max \,(a_{11} - a_{21}, a_{12} - a_{22}, a_{11} - a_{22}) < a_{21} - a_{12}.
\end{equation}
\end{lemma}
\begin{proof}
 We distinguish four cases depending on the sign of~$a_{11} - a_{12}$ and~$a_{22} - a_{21}$. \vspace*{5pt} \\
\noindent{\bf Case 1} -- When~$a_{11} < a_{12}$ and~$a_{22} > a_{21}$, we have
 $$ M_+ = M_- = a_{12} \quad \hbox{and} \quad m_- = m_+ = a_{21} $$
 therefore~\eqref{eq:2-wins-0} holds if and only if~$a_{12} < a_{21}$, which is true by assumption. \vspace*{5pt} \\
\noindent{\bf Case 2} -- When~$a_{11} > a_{12}$ and~$a_{22} > a_{21}$, we have
 $$ M_+ = a_{11}, \quad M_- = a_{11} + (1/N)(a_{12} - a_{11}), \quad m_- = m_+ = a_{21}. $$
 Using some basic algebra, we deduce that~\eqref{eq:2-wins-0} holds if and only if
 $$ (N^2 - N - 1)(a_{11} - a_{21}) < a_{21} - a_{12} \quad \hbox{and} \quad (N - 1)(a_{11} - a_{21}) < a_{21} - a_{12} $$
 therefore~\eqref{eq:2-wins-0} holds if and only if~$(N^2 - N - 1)(a_{11} - a_{21}) < a_{21} - a_{12}$. \vspace*{5pt} \\
\noindent{\bf Case 3} -- Assume that~$a_{11} < a_{12}$ and~$a_{22} < a_{21}$.
 This case can be deduced from the previous one by symmetry exchanging the role of the two strategies, and we find that
 $$ \eqref{eq:2-wins-0} \ \hbox{holds} \quad \hbox{if and only if} \quad (N^2 - N - 1)(a_{12} - a_{22}) < a_{21} - a_{12}. $$
\noindent{\bf Case 4} -- When~$a_{11} > a_{12}$ and~$a_{22} < a_{21}$, it is easier to prove the result graphically and we refer to the phase diagram of Figure~\ref{fig:diagram-2D} for an illustration of some of the arguments of the proof.
 In this case, the minimum and maximum payoffs are given by
 $$ \begin{array}{rcl}
      M_+ = a_{11} & \hbox{and} & M_- = a_{11} + (1/N)(a_{12} - a_{11})  \vspace*{4pt} \\
      m_- = a_{22} & \hbox{and} & m_+ = a_{22} + (1/N)(a_{21} - a_{22}) \end{array} $$
 so the two inequalities in~\eqref{eq:2-wins-0} are respectively equivalent to
\begin{equation}
\label{eq:region-1}
 \begin{array}{rcl}
   (N - 1)^2 \,a_{22} & \n > \n & (N^2 - N - 1) \,a_{11} + a_{12} - (N - 1) \,a_{21}  \vspace*{4pt} \\
   (N - 1)^2 \,a_{11} & \n < \n & (N^2 - N - 1) \,a_{22} + a_{21} - (N - 1) \,a_{12}. \end{array}
\end{equation}
 Since in addition~$a_{11} > a_{12}$ and~$a_{22} < a_{21}$, this specifies two triangles with two common sides, one vertical side and one horizontal side that intersect at point
 $$ p := (a_{12}, a_{21}) \quad \hbox{in the~$a_{11} - a_{22}$ plane}. $$
 For the first inequality in~\eqref{eq:region-1}, the third side of the triangle is the segment line going through point~$p_+$ and with slope~$s_+$ where
\begin{equation}
\label{eq:region-2}
  \begin{array}{rcl}
    p_+ & \n := \n & (a_{21} + (N^2 - N - 1)^{-1} \,(a_{21} - a_{12}), a_{21}) \vspace*{4pt} \\
    s_+ & \n := \n & (N^2 - N - 1)(N - 1)^{-2} = 1 + (N - 2)(N - 1)^{-2} > 1. \end{array}
\end{equation}
 Using some obvious symmetry, one finds that the third side of the triangle specified by the second inequality in~\eqref{eq:region-1} is characterized by the point and slope
\begin{equation}
\label{eq:region-3}
  p_- := (a_{12}, a_{12} - (N^2 - N - 1)^{-1} \,(a_{21} - a_{12})) \quad \hbox{and} \quad s_- = 1/s_+ < 1.
\end{equation}
 Since the segment line connecting~$p_-$ and~$p_+$ has slope one, the triangle~$(p, p_-, p_+)$ is contained in the intersection of the two triangles specified by~\eqref{eq:region-2}--\eqref{eq:region-3}.
 In particular, whenever the payoff coefficients are in this triangle, which is equivalent to
 $$ a_{11} > a_{12} \quad \hbox{and} \quad a_{22} < a_{21} \quad \hbox{and} \quad (N^2 - N - 1)(a_{11} - a_{22}) < a_{21} - a_{12}, $$
 the two inequalities in~\eqref{eq:region-1} hold. \vspace*{5pt} \\
 Since all four cases hold simultaneously when~\eqref{eq:region-0} holds, the proof is complete.
\end{proof} \\ \\
 Theorem~\ref{th:coupling}.b directly follows from Lemmas~\ref{lem:2-wins} and~\ref{lem:region}.


\section{Coupling with a pure growth process}
\label{sec:growth}

\indent This section is devoted to the proof of Theorem~\ref{th:growth} which states that, the other payoffs being fixed, strategy~1 wins whenever the payoff coefficient~$a_{11}$ is sufficiently large.
 The intuition behind this result, which is also the first step of the proof, is to observe that, in the limit as~$a_{11}$ goes to infinity, type~1 players with at least one neighbor of their own type never change their strategy.
 In particular, the set of type~1 players dominates a Richardson model~\cite{richardson_1973}, i.e., a contact process with no death, which obviously implies that strategy~1 wins.
 Using a block construction and the fact that the transition rates are continuous functions of the payoffs, we deduce that the process reaches an equilibrium with a density of type~1 close to one when~$a_{11}$ is finite but large.
 The rest of the proof consists in showing that we can indeed convert the remaining type~2 players, which directly follows from percolation results already established in~\cite{durrett_1992, lanchier_2013}.

\indent To turn our sketch into a rigorous proof, we let~$\zeta_t$ be the~$d$-dimensional Richardson model with parameter~$\mu$.
 In this spin system, each vertex of the~$d$-dimensional integer lattice is either empty or occupied by a particle.
 Each particle produces a new particle which is then sent to a neighbor chosen uniformly at random at rate~$\mu$.
 This results in either an empty site becoming empty or two particles coalescing in case the target site is already occupied.
 In addition, occupied sites remain occupied forever.
 More formally, using the usual notation~0 for empty and 1~for occupied, the transition rates of the process at vertex~$x$ are given by
 $$ c_{0 \to 1} (x, \zeta) = \mu \,f_1 (x, \zeta) \quad \hbox{and} \quad c_{1 \to 0} (x, \zeta) = 0. $$
 In order to compare the process properly rescaled in space and time with oriented site percolation, we also introduce the space-time regions
 $$ B_K := [-K, K]^d \quad \hbox{and} \quad B_K (z) := Kz + B_K \quad \hbox{for all} \quad z \in \Z^d. $$
 Then, we have the following lemma, where the processes under consideration have been identified to the set of vertices in state~1 to lighten the expressions.
\begin{lemma} --
\label{lem:richardson}
 For all~$a, \ep > 0$ there exist~$A, K, c < \infty$ such that
 $$ P \,(B_{2K} \not \subset \xi_t \ \hbox{for some} \ t \in (cK, 2 cK) \,| \,B_K \subset \xi_0) \leq \ep $$
 whenever~$\max \,(a_{21}, a_{22}) \leq a$ and~$a_{11} > A$.
\end{lemma}
\begin{proof}
 To begin with, we note that
 $$ \begin{array}{rclcl}
      \phi (x, \xi) & \n \geq \n & a_{11} \,N_1 (x, \xi) \geq a_{11} & \hbox{when} & x \in \xi \ \hbox{and} \ f_1 (x, \xi) \neq 0 \vspace*{4pt} \\
      \phi (x, \xi) & \n \leq \n & (2M + 1)^d \,\max \,(a_{21}, a_{22})  & \hbox{when} & x \notin \xi \end{array} $$
 from which it follows that, as~$a_{11} \to \infty$,
\begin{equation}
\label{eq:richardson-1}
  \begin{array}{rclcl}
    p_{1 \to 2} (x, \xi) & \hspace*{-5pt} \to \hspace*{-5pt} & 0 & \hbox{when} & f_1 (x, \xi) \neq 0 \vspace*{4pt} \\
    p_{2 \to 1} (x, \xi) & \hspace*{-5pt} \to \hspace*{-5pt} & 1 & \hbox{when} & f_1 (x, \xi) \,f_2 (x, \xi) \neq 0 \ \hbox{for some} \ y \in N_x. \end{array}
\end{equation}
 Looking now at the Richardson model, we observe that, because there is no death and because each particle newly created must be in the neighborhood of its parent, the set of occupied sites satisfies the following connectivity property:
\begin{equation}
\label{eq:richardson-2}
  x \in \zeta_t \quad \hbox{implies that} \quad f_1 (x, \xi_t) \neq 0 \quad \hbox{for all times~$t$}
\end{equation}
 provided this holds at time zero.
 Combining~\eqref{eq:richardson-1}--\eqref{eq:richardson-2} and using that the set~$B_K$ is connected, we deduce that the death-birth process can be coupled with the Richardson model with parameter one in such a way that, for all fixed~$K, c > 0$,
\begin{equation}
\label{eq:richardson-3}
  \begin{array}{l} P \,(B_{2K} \cap \zeta_t \not \subset B_{2K} \cap \xi_t \ \hbox{for some} \ t \in (cK, 2cK) \,| \ \zeta_0 = \xi_0 = B_K) \to 0 \end{array}
\end{equation}
 as~$a_{11} \to \infty$.
 In addition, it directly follows from the shape theorem~\cite{richardson_1973} for the Richardson model that there exists a positive constant~$c > 0$ such that
\begin{equation}
\label{eq:richardson-4}
  \begin{array}{l}
    P \,(B_{2K} \not \subset \zeta_t \ \hbox{for some} \ t \in (cK, 2cK) \,| \ B_K \subset \zeta_0) \vspace*{4pt} \\ \hspace*{25pt} = \
    P \,(B_{2K} \not \subset \zeta_{cK} \,| \ B_K \subset \zeta_0) \leq P \,(B_{2K} \not \subset \zeta_{cK} \,| \ \zeta_0 = \{0 \}) \leq \ep / 2 \end{array}
\end{equation}
 for all~$K$ large.
 Now, fix~$K, c > 0$ such that~\eqref{eq:richardson-4} holds.
 Since the transition rates of the death-birth updating process are continuous with respect to the payoff coefficients and since the space-time region in the event in~\eqref{eq:richardson-3} is finite, there is~$A < \infty$ such that
\begin{equation}
\label{eq:richardson-5}
  \begin{array}{l} P \,(B_{2K} \cap \zeta_t \not \subset B_{2K} \cap \xi_t \ \hbox{for some} \ t \in (cK, 2cK) \,| \ \zeta_0 = \xi_0 = B_K) \leq \ep / 2 \end{array}
\end{equation}
 for all~$a_{11} > A$.
 Combining~\eqref{eq:richardson-4}--\eqref{eq:richardson-5}, we conclude that
 $$ \begin{array}{l}
      P \,(B_{2K} \not \subset \xi_t \ \hbox{for some} \ t \in (cK, 2 cK) \,| \,B_K \subset \xi_0) \vspace*{4pt} \\ \hspace*{25pt} \leq \
      P \,(B_{2K} \not \subset \zeta_t \ \hbox{for some} \ t \in (cK, 2cK) \,| \ B_K \subset \zeta_0) \vspace*{4pt} \\ \hspace*{50pt} + \
      P \,(B_{2K} \cap \zeta_t \not \subset B_{2K} \cap \xi_t \ \hbox{for some} \ t \in (cK, 2cK) \,| \ \zeta_0 = \xi_0 = B_K) \vspace*{4pt} \\ \hspace*{25pt} \leq \
        \ep / 2 + \ep / 2 = \ep \end{array} $$
 for all~$a_{11} > A$. This completes the proof.
\end{proof} \\ \\
 From the lemma, we deduce that, starting from a product measure with a positive density of type~1 players, the density of type~1 at equilibrium is close to one when~$a_{11}$ is large.
 To prove this, we consider as previously the directed graph~$\H$ with vertex set
 $$ H := \{(z, n) \in \Z^d \times \Z_+ : z_1 + z_2 + \cdots + z_d + n \ \hbox{is even} \} $$
 and in which there is an edge~$(z, n) \to (z', n')$ if and only if
 $$ z' = z \pm e_j \ \hbox{for some} \ j = 1, 2, \ldots, d \quad \hbox{and} \quad n' = n + 1. $$
 Then, calling~$(z, n) \in H$ an occupied site when
 $$ B_K (z) \subset \xi_t \quad \hbox{for all} \quad t \in cnK + (0, cK) $$
 it follows from~Lemma~\ref{lem:richardson} that, for all~$\ep > 0$, one can choose~$a_{11}$ large enough so that the set of occupied sites dominates stochastically the set of wet sites in the percolation process
 where sites are closed with probability~$\ep$.
 Since the probability~$\ep$ can be made arbitrarily small, the density of type~1 players at equilibrium can be made arbitrarily close to one.

\indent The last step is to turn the remaining type~2 players into type~1 players.
 To do this, the basic idea is to rely on the lack of percolation of the dry (not wet) sites for a certain oriented site percolation process where sites are closed with a sufficiently small probability~$\ep$.
 The fact that the set of dry sites does not percolate for small positive~$\ep$ is proved in~\cite[section~3]{durrett_1992} for the percolation process described above.
 This result, however, is not sufficient to conclude because the lack of percolation of the dry sites for this percolation process does not imply extinction of strategy~2.
 To solve the problem, we consider oriented site percolation on a directed graph~$\H_+$ that has the same vertex set as before but additional arrows, namely
 $$ \begin{array}{l}
      (z, n) \to (z', n') \quad \hbox{if and only if} \vspace*{4pt} \\ \hspace*{30pt}
      z' = z \pm e_j \ \hbox{for some} \ j = 1, 2, \ldots, d \quad \hbox{and} \quad n' = n + 1 \vspace*{4pt} \\ \hspace*{40pt} \hbox{or} \quad
      z' = z \pm 2 e_j \ \hbox{for some} \ j = 1, 2, \ldots, d \quad \hbox{and} \quad n' = n. \end{array} $$
 The process on~$\H_+$ has the following two key properties:
\begin{enumerate}
 \item As for the process on~$\H$, the dry sites do not percolate if sites are closed with a small enough probability~$\ep > 0$.
  This is proved in~\cite[section~3]{lanchier_2013} following the ideas in~\cite{durrett_1992}. \vspace*{4pt}
 \item Recalling that the death-birth process and the percolation process on~$\H$ are coupled in such a way that the set of occupied sites dominates the set of wet sites, if
  $$ \xi_t (x) = 2 \quad \hbox{for some} \quad (x, t) \in B_K (z) \times (cnK, c (n + 1) K) $$
  then site~$(z, n)$ can be reached by a directed path of dry sites embedded in~$\H_+$.
  This second property is also established in~\cite[section~3]{lanchier_2013}.
  Even though the proof applies to another model, it easily extends to the death-birth process because it only relies on the fact that a type~2 player can only appear in the neighborhood of a type~2 player.
\end{enumerate}
 To deduce extinction of strategy~2, we first fix~$\ep > 0$ small such that the set of dry sites does not percolate for the percolation process on the directed graph~$\H_+$.
 Then, we take~$a_{11}$ large enough so that the set of occupied sites dominates the set of wet sites in the percolation process on the smaller directed graph~$\H$.
 Finally, it follows from the second property above that, because the dry sites do not percolate, the type~2 players do not survive.


\section{The prisoner's dilemma in one dimension}
\label{sec:dilemma}

\indent This section is devoted to the proof of Theorem~\ref{th:dilemma} which focuses on the one-dimensional death-birth process with nearest neighbor interactions.
 First, we explain how~$D_3$ and~$D_4$ in the statement of the theorem are obtained.
 To do so, we let
 $$ \begin{array}{rcl}
      p_i (n_1, n_2) & := & \hbox{the rate at which a type~$i$ player with} \vspace*{0pt} \\ &&
                            \hbox{one type~1 neighbor that has~$n_1$ type~1 neighbors and} \vspace*{0pt} \\ &&
                            \hbox{one type~2 neighbor that has~$n_2$ type~2 neighbors} \vspace*{0pt} \\ &&
                            \hbox{update her strategy} \end{array} $$
 for~$i = 1, 2$ and~$n_1, n_2 = 0, 1, 2$, and note that
 $$ \begin{array}{rcl}
      p_1 (n_1, n_2) & \n = \n & \displaystyle \frac{(2 - n_2) \,a_{21} + n_2 \,a_{22}}{n_1 \,a_{11} + (2 - n_1) \,a_{12} + (2 - n_2) \,a_{21} + n_2 \,a_{22}}  \vspace*{10pt} \\
      p_2 (n_1, n_2) & \n = \n & \displaystyle \frac{n_1 \,a_{11} + (2 - n_1) \,a_{12}}{n_1 \,a_{11} + (2 - n_1) \,a_{12} + (2 - n_2) \,a_{21} + n_2 \,a_{22}}. \end{array} $$
 Now, for the process starting with only~2s to the left of the origin, we let
 $$ X_t := \inf \,\{x \in \Z : \xi_t (x) = 2 \} - 1 \quad \hbox{and} \quad K_t := \inf \,\{x > 0  : \xi_t (x + X_t) = 1 \} $$
 be the position of the rightmost type~1 player with only type~1 players to her left and the distance between this type~1 player and the closest type~1 player to her right.
 Then, letting
 $$ \begin{array}{l} D_j (\xi_t) := \lim_{h \to \,0} \,h^{-1} \,E \,(X_{t + h} - X_t \,| \,\xi_t \ \hbox{and} \ K_t = j) \end{array} $$
 we have the following almost sure estimates
 $$ \begin{array}{ccccccccclll}
           & \hbox{\tiny 1} & \hbox{\tiny 1} & \hbox{\tiny 1} & \hbox{\tiny 2} & \hbox{\tiny 1} \vspace{-6pt} \\
    \cdots & \bullet & \bullet & \bullet & \circ & \bullet & \times  & \times & & D_2 (\xi_t) = D_2 = 2 - p_1 (2, 0) \vspace{-4pt} \\
           & \hbox{\tiny 1} & \hbox{\tiny 1} & \hbox{\tiny 1} & \hbox{\tiny 2} & \hbox{\tiny 2} & \hbox{\tiny 1} \vspace{-6pt} \\
    \cdots & \bullet & \bullet & \bullet & \circ & \circ   & \bullet & \times & & D_3 (\xi_t) = D_3 = p_2 (1, 1) - p_1 (2, 1) \vspace{-4pt} \\
           & \hbox{\tiny 1} & \hbox{\tiny 1} & \hbox{\tiny 1} & \hbox{\tiny 2} & \hbox{\tiny 2} & \hbox{\tiny 2} \vspace{-6pt} \\
    \cdots & \bullet & \bullet & \bullet & \circ & \circ   & \circ   & \times & & D_j (\xi_t) = D_4 = p_2 (1, 2) - p_1 (2, 1) \end{array} $$
 for all~$j > 3$ and where~$\times$ means type~1 or type~2.
 In particular, $D_3$ and~$D_4$ are possible drifts of the interface at~$X_t$ depending on its distance to the next type~1 player.
 Before studying the process starting from general initial configurations, we look at the process starting from configurations that have a finite interval of type~1 players, only type~2 players to the left of this
 interval and infinitely many players of each type to the right.
 In picture, this looks like
 $$ \begin{array}{ccccccccccccccc}
           & \hbox{\tiny 2} & \hbox{\tiny 2} & \hbox{\tiny 2} & \hbox{\tiny 2} & \hbox{\tiny 1} & \hbox{\tiny 1} & \hbox{\tiny 1} & \hbox{\tiny 1} & \hbox{\tiny 1} & \hbox{\tiny 2} & \vspace{-6pt} \\
    \cdots & \circ & \circ & \circ & \circ & \bullet & \bullet & \bullet & \bullet & \bullet & \circ & \times  & \times & \times & \cdots \end{array} $$
 For the process starting from this configuration, we let
 $$ Z_t^- := \inf \,\{x \in \Z : \xi_t (x) = 1 \} \quad \hbox{and} \quad M_t := \inf \,\{x > 0  : \xi_t (x + Z_t^-) = 2 \} $$
 be the position of the leftmost type~1 player and the distance between this type~1 player and the closest type~2 player to her right.
 Then, we have the following result.
\begin{lemma} --
\label{lem:expand}
 Assume that~$a_{22} < a_{21}$ and~$D_3 + D_4 > 0$. Then,
 $$ P \,(M_t > 3 \ \hbox{for all} \ t > 0 \ \hbox{and} \ M_t \to \infty \,| \,M_0 > 3) \,\geq \,c \quad \hbox{for some} \quad c > 0. $$
\end{lemma}
\begin{proof}
 To begin with, we let~$Z_t := e^{- a M_t}$ where~$a > 0$ and
 $$ \begin{array}{l} \Phi_j (a) := \lim_{h \to \,0} \,h^{-1} \,E \,(Z_{t + h} - Z_t \,| \,\xi_t \ \hbox{and} \ M_t > 3 \ \hbox{and} \ K_t = j) \end{array} $$
 for all~$j > 1$.
 To study these quantities, we let
 $$ Z_t^+ := Z_t^- + M_t - 1 \quad \hbox{and} \quad K_t := \inf \,\{x > 0  : \xi_t (x + X_t) = 1 \} $$
 be the right boundary of the type~1 interval starting at~$Z_t^-$ and the distance between this type~1 player and the closest type~1 player to her right.
 Then, we have
\begin{equation}
\label{eq:expand-1}
  \begin{array}{rcl}
  \lim_{h \to \,0} \,h^{-1} \,P \,(Z_{t + h}^+ = Z_t^+ + 1 \,| \,\xi_t \ \hbox{and} \ M_t > 3 \ \hbox{and} \ K_t = 3) & \n = \n & p_2 (1, 1) \vspace*{4pt} \\
  \lim_{h \to \,0} \,h^{-1} \,P \,(Z_{t + h}^+ = Z_t^+ - 1 \,| \,\xi_t \ \hbox{and} \ M_t > 3 \ \hbox{and} \ K_t = 3) & \n = \n & p_1 (2, 1) \end{array}
\end{equation}
 with probability one, while for the left boundary~$Z_t^-$,
\begin{equation}
\label{eq:expand-2}
  \begin{array}{rcl}
  \lim_{h \to \,0} \,h^{-1} \,P \,(Z_{t + h}^- = Z_t^- + 1 \,| \,\xi_t \ \hbox{and} \ M_t > 3) & \n = \n & p_1 (2, 1)  \vspace*{4pt} \\
  \lim_{h \to \,0} \,h^{-1} \,P \,(Z_{t + h}^- = Z_t^- - 1 \,| \,\xi_t \ \hbox{and} \ M_t > 3) & \n = \n & p_2 (1, 2). \end{array}
\end{equation}
 Since~$M_t = Z_t^+ - Z_t^- + 1$, it follows from~\eqref{eq:expand-1}--\eqref{eq:expand-2} that
 $$ \begin{array}{rrl}
    \Phi_3 (a) & \n := \n & \lim_{h \to \,0} \,h^{-1} \,E \,(Z_{t + h} - Z_t \,| \,\xi_t \ \hbox{and} \ M_t > 3 \ \hbox{and} \ K_t = 3) \vspace*{4pt} \\
               & \n  = \n &  Z_t \,(e^{-a} - 1) \,\lim_{h \to \,0} \,h^{-1} \,P \,(M_{t + h} - M_t = 1   \,| \,M_t > 3 \ \hbox{and} \ K_t = 3) \vspace*{4pt} \\ && \hspace*{10pt} + \
                       Z_t \,(e^a    - 1) \,\lim_{h \to \,0} \,h^{-1} \,P \,(M_{t + h} - M_t = - 1 \,| \,M_t > 3 \ \hbox{and} \ K_t = 3) \vspace*{4pt} \\
               & \n  = \n &  (e^{-a} - 1)(p_2 (1, 1) + p_2 (1, 2)) \,Z_t + 2 \,(e^a - 1) \,p_1 (2, 1) \,Z_t \end{array} $$
 therefore, taking the derivative at~$a = 0$, we get
 $$ \Phi_3' (0) = - (p_2 (1, 1) + p_2 (1, 2)) \,Z_t + 2 \,p_1 (2, 1) \,Z_t = - (D_3 + D_4) \,Z_t. $$
 Using the same approach, we prove in general that
 $$ \Phi_j' (0) = - (D_{j \wedge 4} + D_4) \,Z_t \leq - (D_3 + D_4) \,Z_t < 0 \quad \hbox{for all} \quad j > 1 $$
 since~$D_3 < D_4 < D_2$ when~$a_{22} < a_{21}$.
 In particular,
 $$ \Phi_j (b) \leq \Phi_j (0) = 0 \quad \hbox{for some~$b > 0$ fixed from now on} $$
 showing that, as long as~$M_t > 3$, the process~$Z_t$ is a supermartingale with respect to the natural filtration of the death-birth process.
 To conclude, we now apply the optional stopping theorem to this supermartingale using the stopping times
 $$ \tau_3 := \inf \,\{t : M_t \leq 3 \} \quad \hbox{and} \quad \tau_n := \inf \,\{t : M_t \geq n \} \quad \hbox{for all} \quad n > 2. $$
 Since~$T_n := \min (\tau_3, \tau_n)$ is finite, whenever~$M_0 > 3$,
\begin{equation}
\label{eq:expand-3}
  \begin{array}{rcl}
    e^{- 4b} & \n \geq \n & E \,(Z_0) \geq E \,(Z_{T_n}) \vspace*{4pt} \\
             & \n \geq \n & E \,(Z_{T_n} \,| \,T_n = \tau_3) \,P \,(T_n = \tau_3) + E \,(Z_{T_n} \,| \,T_n = \tau_n) \,P \,(T_n = \tau_n) \vspace*{4pt} \\
             & \n \geq \n & e^{- 3b} \,(1 - P \,(T_n = \tau_n)) + e^{- nb} \,P \,(T_n = \tau_n). \end{array}
\end{equation}
 Since the event~$\{T_n = \tau_n \}$ is nonincreasing with respect to~$n$ for the inclusion, we also deduce from the monotone convergence theorem that
\begin{equation}
\label{eq:expand-4}
  \begin{array}{l}
    P \,(M_t > 3 \ \hbox{for all} \ t > 0 \ \hbox{and} \ M_t \to \infty) \vspace*{4pt} \\ \hspace*{40pt} \geq \
    P \,(T_n = \tau_n \ \hbox{for all} \ n > 3) = \lim_{n \to \infty} \,P \,(T_n = \tau_n). \end{array}
\end{equation}
 Combining~\eqref{eq:expand-3}--\eqref{eq:expand-4}, we deduce that
 $$ \begin{array}{l}
     P \,(M_t > 3 \ \hbox{for all} \ t > 0 \ \hbox{and} \ M_t \to \infty \,| \,M_0 > 3) \vspace*{4pt} \\ \hspace*{40pt} \geq \
    \lim_{n \to \infty} \,P \,(T_n = \tau_n) \geq \lim_{n \to \infty} \,(e^{- 3b} - e^{- 4b})(e^{- 3b} - e^{- nb})^{-1} \vspace*{4pt} \\ \hspace*{40pt} \geq \
    (e^{- 3b} - e^{- 4b}) \,e^{3b} = 1 - e^{-b} \end{array} $$
 therefore the lemma holds for~$c := 1 - e^{-b} > 0$.
\end{proof} \\ \\
 To deal with the process starting from product measures, we note that every realization induces a partition of the space-time universe into type 1 and type 2 connected components.
 More precisely, assuming that there is initially a type~2 at the origin, we define the type~2 connected component starting at the origin as
 $$ C_0 := \{(x, t) \in \Z \times \R_+ : \hbox{there is a path~$(0, 0) \to_2 (x, t)$ going forward} \} $$
 where~$(0, 0) \to_2 (x, t)$ means that there are times and vertices
 $$ 0 = t_1 < t_2 < \cdots < t_{n + 1} = t \quad \hbox{and} \quad 0 = x_1, x_2, \ldots, x_n = x $$
 such that the following condition is satisfied:
 $$ (\xi_s (x_j) = 2 \ \ \hbox{for all} \ \ t_j \leq s \leq t_{j + 1}) \ \ \hbox{holds for} \ \ j = 1, 2, \ldots, n. $$
 Then, we have the following lemma.
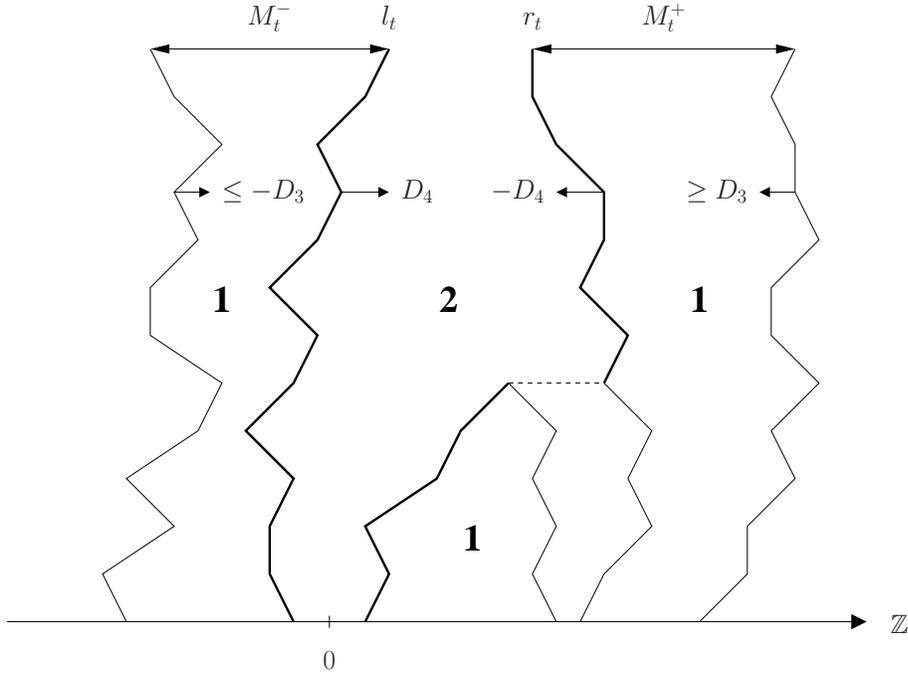
\begin{figure}[t]
\centering
\scalebox{0.50}{\input{dilemma.pstex_t}}
\caption{\upshape{Picture related to the proof of Lemma~\ref{lem:bounded-1}.}}
\label{fig:dilemma}
\end{figure}
\begin{lemma} --
\label{lem:bounded-1}
 Assume that~$a_{22} < a_{21}$ and~$D_3 + D_4 > 0$. Then,
 $$ T := \inf \,\{t > 0 : C_0 \cap (\Z \times (t, \infty)) = \varnothing \} < \infty \quad \hbox{with probability one}. $$
\end{lemma}
\begin{proof}
 We proceed by contradiction, showing that when~$A := \{T = \infty \}$ occurs, its complement occurs with probability one.
 To begin with, note that
\begin{itemize}
 \item on the event~$A$, the type~2 connected component~$C_0$ is unbounded and \vspace*{4pt}
 \item since there are infinitely many type~1 players on both sides of the origin at time zero, this property remains true at all times.
\end{itemize}
 From these two observations, we deduce that
 $$ 0 < \card \,\{x \in \Z : (x, t) \in C_0 \} < \infty \quad \hbox{for all} \quad t \in (0, \infty) $$
 which, in turn, implies that the left boundary~$l_t$ and right boundary~$r_t$ of the type~2 connected component satisfy the following properties at all times:
 $$ \begin{array}{l}
      - \infty < l_t := \inf \,\{x \in \Z : (x, t) \in C_0 \} \vspace*{4pt} \\ \hspace*{100pt} \leq \sup \,\{x \in \Z : (x, t) \in C_0 \} =: r_t < \infty. \end{array} $$
 We also observe that
\begin{equation}
\label{eq:bounded-1}
  \begin{array}{rcl}
    M_t^- & \n := \n & \inf \,\{x > 0 : \xi_t (l_t - x) = 2 \} > 1  \vspace*{4pt} \\
    M_t^+ & \n := \n & \inf \,\{x > 0 : \xi_t (r_t + x) = 2 \} > 1. \end{array}
\end{equation}
 Figure~\ref{fig:dilemma} gives an illustration of these processes.
 Now, the evolution rules of the death-birth updating process clearly imply that there exist~$c_1, c_2 > 0$ such that
\begin{equation}
\label{eq:bounded-2}
  \begin{array}{rcl}
    P \,(\min \,(M_{t + 1}^-, M_{t + 1}^+) > 3 \ | \,\min \,(M_t^-, M_t^+) > 1) & \n \geq \n & c_1 \vspace*{4pt} \\
                                         P \,(T < t + 1 \ | \ r_t - l_t \leq 3) & \n \geq \n & c_2 \end{array}
\end{equation}
 while it follows from Lemma~\ref{lem:expand} that
\begin{equation}
\label{eq:bounded-3}
  \begin{array}{l}
    P \,(\inf \,\{t > s : \min \,(M_t^-, M_t^+) \leq 3 \} \vspace*{4pt} \\ \hspace*{60pt}
                          \geq \inf \,\{t > s : r_t - l_t \leq 3 \} \ | \,\min \,(M_s^-, M_s^+) > 3) \geq c^2 > 0. \end{array}
\end{equation}
 Combining~\eqref{eq:bounded-1}--\eqref{eq:bounded-3}, we conclude that
 $$ \begin{array}{rcl}
      A = \{T = \infty \} \ \hbox{occurs} & \hbox{implies that} & \min \,(M_t^-, M_t^+) > 1 \ \hbox{at all times} \vspace*{3pt} \\
                                          & \hbox{implies that} & \min \,(M_t^-, M_t^+) > 3 \ \hbox{infinitely often} \vspace*{3pt} \\
                                          & \hbox{implies that} & r_t - l_t \leq 3 \ \hbox{infinitely often} \vspace*{3pt} \\
                                          & \hbox{implies that} & T < \infty \ \hbox{with probability one}. \end{array} $$
 This completes the proof.
\end{proof}
\begin{lemma} --
\label{lem:bounded-2}
 Assume that~$a_{22} > a_{21}$ and~$D_4 > 0$. Then,
 $$ T := \inf \,\{t > 0 : C_0 \cap (\Z \times (t, \infty)) = \varnothing \} < \infty \quad \hbox{with probability one}. $$
\end{lemma}
\begin{proof}
 First, we note that the conclusion of Lemma~\ref{lem:expand} holds as well under the assumptions of the present lemma, since whenever~$a_{22} > a_{21}$ we have
 $$ D_4 < D_3 < D_2 \quad \hbox{and} \quad D_{j \wedge 4} + D_4 \geq 2 D_4 \ \ \hbox{for all} \ \ j > 1. $$
 In particular, the lemma follows by repeating the proof of Lemma~\ref{lem:bounded-1}.
\end{proof} \\ \\
 Under the assumptions of Theorem~\ref{th:dilemma} and conditional on the origin being initially occupied by a type~2 player, it follows from Lemmas~\ref{lem:bounded-1}--\ref{lem:bounded-2} that
 the type~2 connected starting at the origin is bounded with probability one.
 Since in addition the number of players is countable,
 $$ \begin{array}{rcl}
      P \,(\hbox{strategy~2 wins}) & \n = \n & P \,(\hbox{$\xi_0 (x) = 2$ and~$T_x = \infty$ for some~$x \in \Z$}) \vspace*{4pt} \\
                                   & \n \leq \n & \sum_{x : \xi_0 (x) = 2} \,P \,(T_x = \infty) = 0 \end{array} $$
 where~$T_x$ denotes the time at which the connected component starting at vertex~$x$ dies out, as defined in the statement of Lemmas~\ref{lem:bounded-1}--\ref{lem:bounded-2}.
 This proves Theorem~\ref{th:dilemma}.


\end{document}

%% file: diagram-2D.pstex_t
\begin{picture}(0,0)%
\includegraphics{diagram-2D.pstex}%
\end{picture}%
\setlength{\unitlength}{3947sp}%
\begingroup\makeatletter\ifx\SetFigFont\undefined%
\gdef\SetFigFont#1#2#3#4#5{%
  \reset@font\fontsize{#1}{#2pt}%
  \fontfamily{#3}\fontseries{#4}\fontshape{#5}%
  \selectfont}%
\fi\endgroup%
\begin{picture}(8235,8397)(-794,-7126)
\put(601,-5986){\makebox(0,0)[b]{\smash{{\SetFigFont{14}{16.8}{\rmdefault}{\bfdefault}{\updefault}TH \ref{th:coex}}}}}
\put(  1,1064){\makebox(0,0)[b]{\smash{{\SetFigFont{14}{16.8}{\rmdefault}{\bfdefault}{\updefault}$a_{22}$}}}}
\put(7426,-6436){\makebox(0,0)[lb]{\smash{{\SetFigFont{14}{16.8}{\rmdefault}{\bfdefault}{\updefault}$a_{11}$}}}}
\put(4876,-2461){\makebox(0,0)[lb]{\smash{{\SetFigFont{14}{16.8}{\rmdefault}{\bfdefault}{\updefault}$p_+$}}}}
\put(3226,-4111){\makebox(0,0)[lb]{\smash{{\SetFigFont{14}{16.8}{\rmdefault}{\bfdefault}{\updefault}$p_-$}}}}
\put(6901,-3211){\makebox(0,0)[rb]{\smash{{\SetFigFont{14}{16.8}{\rmdefault}{\bfdefault}{\updefault}$a_{22} = a_{12}$}}}}
\put(6901,-2011){\makebox(0,0)[rb]{\smash{{\SetFigFont{14}{16.8}{\rmdefault}{\bfdefault}{\updefault}$a_{22} = a_{21}$}}}}
\put(3001,1064){\makebox(0,0)[b]{\smash{{\SetFigFont{14}{16.8}{\rmdefault}{\bfdefault}{\updefault}$a_{11} = a_{12}$}}}}
\put(4201,1064){\makebox(0,0)[b]{\smash{{\SetFigFont{14}{16.8}{\rmdefault}{\bfdefault}{\updefault}$a_{11} = a_{21}$}}}}
\put(-599,-4861){\rotatebox{90.0}{\makebox(0,0)[b]{\smash{{\SetFigFont{14}{16.8}{\rmdefault}{\bfdefault}{\updefault}2 altruistic}}}}}
\put(-599,-1261){\rotatebox{90.0}{\makebox(0,0)[b]{\smash{{\SetFigFont{14}{16.8}{\rmdefault}{\bfdefault}{\updefault}2 selfish}}}}}
\put(2101,-7111){\makebox(0,0)[b]{\smash{{\SetFigFont{14}{16.8}{\rmdefault}{\bfdefault}{\updefault}1 altruistic}}}}
\put(5701,-7111){\makebox(0,0)[b]{\smash{{\SetFigFont{14}{16.8}{\rmdefault}{\bfdefault}{\updefault}1 selfish}}}}
\put(1501,-2836){\makebox(0,0)[b]{\smash{{\SetFigFont{14}{16.8}{\rmdefault}{\bfdefault}{\updefault}TH \ref{th:coupling}.b}}}}
\put(3601,-736){\makebox(0,0)[b]{\smash{{\SetFigFont{14}{16.8}{\rmdefault}{\bfdefault}{\updefault}TH \ref{th:coupling}.b}}}}
\put(6301,-5761){\makebox(0,0)[b]{\smash{{\SetFigFont{14}{16.8}{\rmdefault}{\bfdefault}{\updefault}TH \ref{th:coupling}.a}}}}
\put(2776,-1936){\makebox(0,0)[rb]{\smash{{\SetFigFont{14}{16.8}{\rmdefault}{\bfdefault}{\updefault}$p$}}}}
\end{picture}%

%% file: diagram-1D.pstex_t
\begin{picture}(0,0)%
\includegraphics{diagram-1D.pstex}%
\end{picture}%
\setlength{\unitlength}{3947sp}%
\begingroup\makeatletter\ifx\SetFigFont\undefined%
\gdef\SetFigFont#1#2#3#4#5{%
  \reset@font\fontsize{#1}{#2pt}%
  \fontfamily{#3}\fontseries{#4}\fontshape{#5}%
  \selectfont}%
\fi\endgroup%
\begin{picture}(8235,8397)(-794,-7126)
\put(4051,-5086){\makebox(0,0)[b]{\smash{{\SetFigFont{12}{14.4}{\rmdefault}{\bfdefault}{\updefault}dilemma}}}}
\put(4051,-4861){\makebox(0,0)[b]{\smash{{\SetFigFont{12}{14.4}{\rmdefault}{\bfdefault}{\updefault}prisoner's}}}}
\put(4051,-5311){\makebox(0,0)[b]{\smash{{\SetFigFont{12}{14.4}{\rmdefault}{\bfdefault}{\updefault}triangle}}}}
\put(2401,-7111){\makebox(0,0)[b]{\smash{{\SetFigFont{14}{16.8}{\rmdefault}{\bfdefault}{\updefault}1 altruistic}}}}
\put(6001,-7111){\makebox(0,0)[b]{\smash{{\SetFigFont{14}{16.8}{\rmdefault}{\bfdefault}{\updefault}1 selfish}}}}
\put(-599,-5161){\rotatebox{90.0}{\makebox(0,0)[b]{\smash{{\SetFigFont{14}{16.8}{\rmdefault}{\bfdefault}{\updefault}2 altruistic}}}}}
\put(-599,-1561){\rotatebox{90.0}{\makebox(0,0)[b]{\smash{{\SetFigFont{14}{16.8}{\rmdefault}{\bfdefault}{\updefault}2 selfish}}}}}
\put(  1,1064){\makebox(0,0)[b]{\smash{{\SetFigFont{14}{16.8}{\rmdefault}{\bfdefault}{\updefault}$a_{22}$}}}}
\put(7426,-6436){\makebox(0,0)[lb]{\smash{{\SetFigFont{14}{16.8}{\rmdefault}{\bfdefault}{\updefault}$a_{11}$}}}}
\put(4801,1064){\makebox(0,0)[b]{\smash{{\SetFigFont{14}{16.8}{\rmdefault}{\bfdefault}{\updefault}$a_{11} = a_{21}$}}}}
\put(2401,1064){\makebox(0,0)[b]{\smash{{\SetFigFont{14}{16.8}{\rmdefault}{\bfdefault}{\updefault}$a_{11} = a_{12}$}}}}
\put(1201,-1411){\makebox(0,0)[b]{\smash{{\SetFigFont{14}{16.8}{\rmdefault}{\bfdefault}{\updefault}$a_{22} = a_{21}$}}}}
\put(1201,-3811){\makebox(0,0)[b]{\smash{{\SetFigFont{14}{16.8}{\rmdefault}{\bfdefault}{\updefault}$a_{22} = a_{12}$}}}}
\end{picture}%

%% file: coex.pstex_t
\begin{picture}(0,0)%
\includegraphics{coex.pstex}%
\end{picture}%
\setlength{\unitlength}{3947sp}%
\begingroup\makeatletter\ifx\SetFigFont\undefined%
\gdef\SetFigFont#1#2#3#4#5{%
  \reset@font\fontsize{#1}{#2pt}%
  \fontfamily{#3}\fontseries{#4}\fontshape{#5}%
  \selectfont}%
\fi\endgroup%
\begin{picture}(9048,9747)(-32,-8875)
\put(3601,-8761){\makebox(0,0)[b]{\smash{{\SetFigFont{20}{24.0}{\familydefault}{\mddefault}{\updefault}$-M/2$}}}}
\put(5401,-8761){\makebox(0,0)[b]{\smash{{\SetFigFont{20}{24.0}{\familydefault}{\mddefault}{\updefault}$M/2$}}}}
\put(  1,-8761){\makebox(0,0)[b]{\smash{{\SetFigFont{20}{24.0}{\familydefault}{\mddefault}{\updefault}$-5M/2$}}}}
\put(5551,-4411){\makebox(0,0)[lb]{\smash{{\SetFigFont{20}{24.0}{\familydefault}{\mddefault}{\updefault}$B_5 = B$}}}}
\put(151,-7861){\makebox(0,0)[lb]{\smash{{\SetFigFont{20}{24.0}{\familydefault}{\mddefault}{\updefault}$K_{5/2}$}}}}
\put(9001,-8761){\makebox(0,0)[b]{\smash{{\SetFigFont{20}{24.0}{\familydefault}{\mddefault}{\updefault}$5M/2$}}}}
\put(1801,-8761){\makebox(0,0)[b]{\smash{{\SetFigFont{20}{24.0}{\familydefault}{\mddefault}{\updefault}$-3M/2$}}}}
\put(7201,-8761){\makebox(0,0)[b]{\smash{{\SetFigFont{20}{24.0}{\familydefault}{\mddefault}{\updefault}$3M/2$}}}}
\put(1051,239){\makebox(0,0)[lb]{\smash{{\SetFigFont{20}{24.0}{\familydefault}{\mddefault}{\updefault}$B_0$}}}}
\put(1951,-661){\makebox(0,0)[lb]{\smash{{\SetFigFont{20}{24.0}{\familydefault}{\mddefault}{\updefault}$B_1$}}}}
\put(2851,-1561){\makebox(0,0)[lb]{\smash{{\SetFigFont{20}{24.0}{\familydefault}{\mddefault}{\updefault}$B_2$}}}}
\put(3751,-2461){\makebox(0,0)[lb]{\smash{{\SetFigFont{20}{24.0}{\familydefault}{\mddefault}{\updefault}$B_3$}}}}
\put(5551,-3361){\makebox(0,0)[lb]{\smash{{\SetFigFont{20}{24.0}{\familydefault}{\mddefault}{\updefault}$K_{1/2}$}}}}
\put(4651,-3361){\makebox(0,0)[lb]{\smash{{\SetFigFont{20}{24.0}{\familydefault}{\mddefault}{\updefault}$B_4$}}}}
\end{picture}%

%% file: concave.pstex_t
\begin{picture}(0,0)%
\includegraphics{concave.pstex}%
\end{picture}%
\setlength{\unitlength}{3947sp}%
\begingroup\makeatletter\ifx\SetFigFont\undefined%
\gdef\SetFigFont#1#2#3#4#5{%
  \reset@font\fontsize{#1}{#2pt}%
  \fontfamily{#3}\fontseries{#4}\fontshape{#5}%
  \selectfont}%
\fi\endgroup%
\begin{picture}(12695,6174)(-454,-5125)
\put(  1,-5011){\makebox(0,0)[b]{\smash{{\SetFigFont{20}{24.0}{\familydefault}{\mddefault}{\updefault}0}}}}
\put(901,-5011){\makebox(0,0)[b]{\smash{{\SetFigFont{20}{24.0}{\familydefault}{\mddefault}{\updefault}1}}}}
\put(5401,-5011){\makebox(0,0)[b]{\smash{{\SetFigFont{20}{24.0}{\familydefault}{\mddefault}{\updefault}$N$}}}}
\put(6601,-5011){\makebox(0,0)[b]{\smash{{\SetFigFont{20}{24.0}{\familydefault}{\mddefault}{\updefault}0}}}}
\put(12001,-5011){\makebox(0,0)[b]{\smash{{\SetFigFont{20}{24.0}{\familydefault}{\mddefault}{\updefault}$N$}}}}
\put(11101,-5011){\makebox(0,0)[b]{\smash{{\SetFigFont{20}{24.0}{\familydefault}{\mddefault}{\updefault}$N - 1$}}}}
\put(12226,-136){\makebox(0,0)[lb]{\smash{{\SetFigFont{20}{24.0}{\familydefault}{\mddefault}{\updefault}$1 - \ep$}}}}
\put(6376,764){\makebox(0,0)[rb]{\smash{{\SetFigFont{20}{24.0}{\familydefault}{\mddefault}{\updefault}1}}}}
\put(-224,764){\makebox(0,0)[rb]{\smash{{\SetFigFont{20}{24.0}{\familydefault}{\mddefault}{\updefault}1}}}}
\put(-224,-3736){\makebox(0,0)[rb]{\smash{{\SetFigFont{20}{24.0}{\familydefault}{\mddefault}{\updefault}$\ep$}}}}
\put(2251,-511){\makebox(0,0)[rb]{\smash{{\SetFigFont{20}{24.0}{\familydefault}{\mddefault}{\updefault}$g_1 (z)$}}}}
\put(3151,-1711){\makebox(0,0)[lb]{\smash{{\SetFigFont{20}{24.0}{\familydefault}{\mddefault}{\updefault}$h_1 (z)$}}}}
\put(8776,-2236){\makebox(0,0)[rb]{\smash{{\SetFigFont{20}{24.0}{\familydefault}{\mddefault}{\updefault}$h_2 (z)$}}}}
\put(9676,-3436){\makebox(0,0)[lb]{\smash{{\SetFigFont{20}{24.0}{\familydefault}{\mddefault}{\updefault}$g_2 (z)$}}}}
\end{picture}%

%% file: dilemma.pstex_t
\begin{picture}(0,0)%
\includegraphics{dilemma.pstex}%
\end{picture}%
\setlength{\unitlength}{3947sp}%
\begingroup\makeatletter\ifx\SetFigFont\undefined%
\gdef\SetFigFont#1#2#3#4#5{%
  \reset@font\fontsize{#1}{#2pt}%
  \fontfamily{#3}\fontseries{#4}\fontshape{#5}%
  \selectfont}%
\fi\endgroup%
\begin{picture}(11127,8406)(11989,-6976)
\put(16051,-6961){\makebox(0,0)[b]{\smash{{\SetFigFont{20}{24.0}{\familydefault}{\mddefault}{\updefault}0}}}}
\put(23101,-6511){\makebox(0,0)[lb]{\smash{{\SetFigFont{20}{24.0}{\familydefault}{\mddefault}{\updefault}$\Z$}}}}
\put(16801,1139){\makebox(0,0)[b]{\smash{{\SetFigFont{20}{24.0}{\familydefault}{\mddefault}{\updefault}$l_t$}}}}
\put(18601,1139){\makebox(0,0)[b]{\smash{{\SetFigFont{20}{24.0}{\familydefault}{\mddefault}{\updefault}$r_t$}}}}
\put(20251,1139){\makebox(0,0)[b]{\smash{{\SetFigFont{20}{24.0}{\familydefault}{\mddefault}{\updefault}$M_t^+$}}}}
\put(15301,1139){\makebox(0,0)[b]{\smash{{\SetFigFont{20}{24.0}{\familydefault}{\mddefault}{\updefault}$M_t^-$}}}}
\put(14701,-1036){\makebox(0,0)[lb]{\smash{{\SetFigFont{20}{24.0}{\familydefault}{\mddefault}{\updefault}$\leq - D_3$}}}}
\put(21301,-1036){\makebox(0,0)[rb]{\smash{{\SetFigFont{20}{24.0}{\familydefault}{\mddefault}{\updefault}$\geq D_3$}}}}
\put(16951,-1036){\makebox(0,0)[lb]{\smash{{\SetFigFont{20}{24.0}{\familydefault}{\mddefault}{\updefault}$D_4$}}}}
\put(18751,-1036){\makebox(0,0)[rb]{\smash{{\SetFigFont{20}{24.0}{\familydefault}{\mddefault}{\updefault}$- D_4$}}}}
\end{picture}%